\documentclass[11pt]{article}

\usepackage[utf8]{inputenc}
\usepackage[dvips]{graphicx}
\usepackage{calc}
\usepackage{color}
\usepackage{amsmath}
\usepackage{amssymb}
\usepackage{amscd}
\usepackage{amsthm}
\usepackage{amsbsy}
\usepackage{delarray}
\usepackage{enumerate}
\usepackage[T1]{fontenc}
\usepackage{inputenc}
\usepackage{enumerate}
\usepackage{hyperref}

\begin{document}



\setlength{\parindent}{5mm}
\renewcommand{\leq}{\leqslant}
\renewcommand{\geq}{\geqslant}
\newcommand{\N}{\mathbb{N}}
\newcommand{\sph}{\mathbb{S}}
\newcommand{\Z}{\mathbb{Z}}
\newcommand{\R}{\mathbb{R}}
\newcommand{\C}{\mathbb{C}}
\newcommand{\F}{\mathbb{F}}
\newcommand{\g}{\mathfrak{g}}
\newcommand{\h}{\mathfrak{h}}
\newcommand{\K}{\mathbb{K}}
\newcommand{\RN}{\mathbb{R}^{2n}}
\newcommand{\ci}{c^{\infty}}
\newcommand{\derive}[2]{\frac{\partial{#1}}{\partial{#2}}}
\renewcommand{\S}{\mathbb{S}}
\renewcommand{\H}{\mathbb{H}}
\newcommand{\eps}{\varepsilon}
\newcommand{\chzrel}{c_{\mathrm{LR}}}
\newcommand{\Href}{H_{\mathrm{ref}}}
\newcommand{\Jref}{J_{\mathrm{ref}}}
\newcommand{\p}{\textbf{p}}
\newcommand{\diam}{\mathrm{diam}}

\theoremstyle{plain}
\newtheorem{theo}{Theorem}
\newtheorem{prop}[theo]{Proposition}
\newtheorem{lemma}[theo]{Lemma}
\newtheorem{definition}[theo]{Definition}
\newtheorem*{notation*}{Notation}
\newtheorem*{notations*}{Notations}
\newtheorem{corol}[theo]{Corollary}
\newtheorem{conj}[theo]{Conjecture}
\newtheorem*{conj*}{Conjecture}
\newtheorem*{claim*}{Claim}
\newtheorem{claim}[theo]{Claim}

\newenvironment{demo}[1][]{\addvspace{8mm} \emph{Proof #1.
    ~~}}{~~~$\Box$\bigskip}

\newlength{\espaceavantspecialthm}
\newlength{\espaceapresspecialthm}
\setlength{\espaceavantspecialthm}{\topsep} \setlength{\espaceapresspecialthm}{\topsep}

\newtheorem{exple}[theo]{Example}
\renewcommand{\theexple}{}
\newenvironment{example}{\begin{exple}\rm }{\hfill $\blacktriangleleft$\end{exple}}

\newtheorem{quest}[theo]{Question}
\renewcommand{\thequest}{}
\newenvironment{question}{\begin{quest}\it }{\end{quest}}

\newenvironment{remark}[1][]{\refstepcounter{theo} 
\vskip \espaceavantspecialthm \noindent \textsc{Remark~\thetheo
#1.} }%
{\vskip \espaceapresspecialthm}

\def\bb#1{\mathbb{#1}} \def\m#1{\mathcal{#1}}

\def\del{\partial}
\def\co{\colon\thinspace}
\def\Homeo{\mathrm{Homeo}}
\def\Hameo{\mathrm{Hameo}}
\def\Diffeo{\mathrm{Diffeo}}
\def\Symp{\mathrm{Symp}}
\def\Sympeo{\mathrm{Sympeo}}
\def\id{\mathrm{Id}}
\newcommand{\norm}[1]{||#1||}
\def\Ham{\mathrm{Ham}}
\def\lagham#1{\mathcal{L}^\mathrm{Ham}({#1})}
\def\Hamtilde{\widetilde{\mathrm{Ham}}}
\def\cOlag#1{\mathrm{Sympeo}({#1})}
\def\Crit{\mathrm{Crit}}
\def\Spec{\mathrm{Spec}}
\def\osc{\mathrm{osc}}
\def\Cal{\mathrm{Cal}}
\def\propT{(\mathcal{T})}

\definecolor{sobhan}{rgb}{0,.6,0}\newcommand{\sobhan}{\color{sobhan}}
\definecolor{lev}{rgb}{0.773,0.294,0.549}\newcommand{\lev}{\color{lev}} 
\definecolor{vincent}{rgb}{0,0,1}\newcommand{\vincent}{\color{vincent}}

\def\fnsymbol{$\alpha$}
\renewcommand{\thefootnote}{$\alpha$}

\title{A $C^0$ counterexample to the Arnold conjecture}
\author{Lev Buhovsky$^\alpha$, Vincent Humili{\`e}re, Sobhan Seyfaddini}
\footnotetext[1]{This author also uses the spelling ``Buhovski''
for his family name.}
\date{\today}

\renewcommand{\thefootnote}{\arabic{footnote}}

\maketitle
\begin{abstract} 
The Arnold conjecture states that a Hamiltonian diffeomorphism of a closed and connected symplectic manifold $(M, \omega)$ must have at least as many fixed points as the minimal number of critical points of a smooth function on $M$.
 
  It is well known that the Arnold conjecture holds for Hamiltonian \emph{homeomorphisms} of closed  symplectic surfaces.  The goal of this paper is to provide a counterexample to the Arnold conjecture for Hamiltonian homeomorphisms in dimensions four and higher.
  
  More precisely, we prove that every closed and connected symplectic manifold of dimension at least four admits a Hamiltonian homeomorphism with a single fixed point.
\end{abstract}

\tableofcontents


\section{Introduction and main results}
\subsection{The Arnold conjecture}

Let $(M, \omega)$ denote a closed and connected symplectic manifold.  This paper is concerned with the celebrated  conjecture of Arnold  on fixed points of Hamiltonian diffeomorphisms.  

\begin{conj*}[Arnold]
A Hamiltonian diffeomorphism of $M$ must have at least as many fixed points as the minimal number of critical points of a smooth function on $M$.
\end{conj*}

What makes this conjecture so remarkable is the large number of fixed points predicted by it.  This is often interpreted as a manifestation of symplectic rigidity.  In contrast to Arnold's conjecture, the classical Lefschetz fixed-point theorem  cannot predict the existence of more than one fixed point for a general diffeomorphism. Ever since its inception, this simple and beautiful conjecture has been a powerful driving force in the development of symplectic topology. The most important breakthrough towards a solution of this conjecture came with Floer's invention of what is now called \emph{Hamiltonian Floer homology} which established a variant of the Arnold conjecture on a large class of symplectic manifolds \cite{floer86, floer88, floer89}. The above version of the Arnold conjecture has been established on symplectically aspherical\footnote{$M$ is said to be symplectically aspherical if $\omega$ and $c_1$, the first Chern class of $M$, both vanish on $\pi_2(M)$.}  manifolds by Rudyak and Oprea in \cite{rudyak-oprea} who built on earlier works of Floer \cite{floer89b} and Hofer \cite{hofer}. 
We should mention that prior to the discovery of Floer homology,  the  Arnold conjecture was proven by Eliashberg \cite{eliashberg} on closed surfaces, by  Conley and Zehnder \cite{conley-zehnder} on higher dimensional tori, and by Fortune and Weinstein \cite{fortune, FW} on complex projective spaces.

\subsection{The Arnold conjecture and Hamiltonian homeomorphisms}
Throughout this paper we will denote by $\Symp(M, \omega)$ and $\Ham(M, \omega)$ the groups of symplectic and Hamiltonian  diffeomorphisms of $(M, \omega)$, respectively.  As is nowadays standard, we call \emph{symplectic homeomorphism} any homeomorphism which can be written as a uniform limit of symplectic diffeomorphisms; the set of all symplectic homeomorphisms is denoted by $\Sympeo(M, \omega)$; see Section \ref{sec:sympl-hamilt-home}.

As a first attempt at defining \emph{Hamiltonian homeomorphisms}, we will say that a homeomorphism $\phi$ of $M$ is a \emph{Hamiltonian homeomorphism} if it can be written as a uniform limit of Hamiltonian diffeomorphisms.  This class of homeomorphisms has been studied very extensively, from a dynamical point of view, in the case of closed surfaces\footnote{This is precisely the class of area preserving homeomorphisms with vanishing mean rotation vector.}. For example, Matsumoto \cite{matsumoto}, building on an earlier paper of Franks \cite{franks}, has proven that Hamiltonian homeomorphisms of surfaces satisfy the Arnold conjecture.  An important development in the study of Hamiltonian homeomorphisms of surfaces has been Le Calvez's theory of transverse foliations \cite{lecalvez05} which has not only proven the Arnold conjecture but also the Conley conjecture on periodic points of these homeomorphisms \cite{lecalvez06}. 

In striking contrast to the rich theory in dimension two, there are virtually no  results on fixed point theory of Hamiltonian homeomorphisms in higher dimensions.  Indeed, none of the powerful tools of surface dynamics seem to generalize in an obvious manner to dimensions higher than two. Our first theorem proves that in fact one can not hope to prove the Arnold conjecture in higher dimensions.

\begin{theo}
Every closed and connected symplectic manifold of dimension at least 4 admits a Hamiltonian homeomorphism with a single fixed point.
\end{theo}

This theorem might suggest that, in dimensions higher than two, one should search for a different notion of Hamiltonian homeomorphisms. Indeed, such notion does exist within the field of continuous, or $C^0$,  symplectic topology.  Motivated in part by developing a continuous analogue of smooth Hamiltonian dynamics, M\"uller and Oh have suggested an alternative, more restrictive, definition for Hamiltonian homeomorphisms; see Section \ref{sec:sympl-hamilt-home} for the precise definition.  From this point onward by Hamiltonian homeomorphisms we will mean those homeomorphisms of $M$ prescribed by Definition \ref{def:hameo}.  We denote the set of all Hamiltonian homeomorphisms by $\Hameo(M, \omega)$. 

The group $\Hameo(M, \omega)$ has met some success.  Indeed, recent results in $C^0$-symplectic topology \cite{HLS12, HLS13, HLS14} have demonstrated that Hamiltonian homeomorphisms inherit some of the important dynamical properties of smooth Hamiltonian diffeomorphisms; see Theorem \ref{theo:coiso-unique}.  Furthermore, they have played a key role in the development of $C^0$-symplectic topology over the past several years.  However, our main theorem proves that the Arnold conjecture is not true for this notion of Hamiltonian homeomorphisms either.  In fact, as we will explain below, it shows that there is no hope for proving the Arnold conjecture,  as formulated above, for any alternate definition of Hamiltonian homeomorphisms which satisfies a minimal set of requirements.

\begin{theo}[Main Theorem] \label{theo:main}
Let $(M, \omega)$ denote a closed and connected symplectic manifold of dimension at least $4$.  There exists $f \in \Hameo(M, \omega)$ with a single fixed point.  Furthermore, $f$ can be chosen to satisfy either of the following additional properties.
\begin{enumerate}
\item Let $\mathcal{H}$ be a normal subgroup of $\Sympeo(M, \omega)$ which contains $\Ham(M, \omega)$ as a proper subset.  Then, $f \in \mathcal{H}$.
\item Let $p$ denote the unique fixed point of $f$.  Then, $f$ is a symplectic diffeomorphism of $M \setminus \{p\}$.
\end{enumerate}
\end{theo}

A few remarks are in order.  First, we should point out that every Hamiltonian homeomorphism possesses at least one fixed point. This is because a Hamiltonian homeomorphism is by definition a uniform limit of Hamiltonian diffeomorphisms and it is a non-trivial fact that a Hamiltonian diffeomorphism has at least one fixed point. \footnote{ This fact is an immediate consequence of Floer's proof of the Arnold conjecture; see also \cite{Grom}.} 

Second, we remark that it is well known that $\Ham(M, \omega)$ is a normal subgroup of $\Symp(M, \omega)$. Hence, it is reasonable to expect that any alternative candidate, say $\mathcal{H}$, for the group of Hamiltonian homeomorphisms should contain $Ham(M, \omega)$ and be a normal subgroup of $\Sympeo(M, \omega)$.  It is indeed the case that $\Hameo(M, \omega) \unlhd \Sympeo(M, \omega)$.  Therefore, the first property in the above theorem states that there is no hope of proving the Arnold conjecture for any alternate definition of Hamiltonian homeomorphisms.

 Lastly, with regards to the second property, we point out that it is natural to expect  $ f $ to have at least one non-smooth point. Indeed, since Hamiltonian Floer homology predicts that a Hamiltonian diffeomorphism can never have as few as one fixed point, our homeomorphism $f$ must necessarily be non-smooth on any symplectic manifold $(M, \omega)$ with  the property \footnote{ It can be shown that this property holds for closed symplectic surfaces, as well as for the standard $\mathbb{C}P^2$ and monotone $S^2 \times S^2$.} that $ \Hameo(M,\omega) \cap \mathrm{Diff}(M) = \Ham(M,\omega)$.

\subsection{Does there exist a fixed point theory for Hamiltonian homeomorphisms?} \label{sec:disc-does-there} 
 In Gromov's view \cite{Grom2}, symplectic topology is enriched by a beautiful interplay between rigidity and flexibility.  Recent results, such as \cite{ops, HLS13, BuOp},  have demonstrated that this contrast between rigidity and flexibility permeates, in a surprising fashion, to $C^0$ symplectic topology as well. Symplectic rigidity manifests itself when  symplectic phenomena survives under $C^0$ limits; see  \cite{cardin-viterbo, entov-polterovich, buhovsky, ops, HLS13} for some examples.  On the other hand, there exist instances where passage to $C^0$ limits results in spectacular loss of rigidity and prevalence of flexibility; see \cite{BuOp} for an example.

The main theorem of our paper tells us that fixed points of Hamiltonian diffeomorphisms become completely flexible under $C^0$ limits.  It is  interesting to contrast this prevalence of flexibility with the strong rigidity results of Franks \cite{franks}, Matsumoto \cite{matsumoto}, and Le Calvez \cite{lecalvez05, lecalvez06}
in the two-dimensional setting.  Given the main result of this article, one might conclude that there is no hope of developing a sensible fixed point theory for any notion of Hamiltonian homeomorphisms in dimensions greater than two.  However, there exist some interesting open questions which remain unanswered.  
  
  The most prominent open question is that of the Conley conjecture which in its simplest form states that a Hamiltonian diffeomorphism on an aspherical symplectic manifold has infinitely many periodic points.  This conjecture was proven by Hingston \cite{hingston} on tori and Ginzburg \cite{ginzburg} in the more general setting.  As mentioned earlier, the Conley conjecture has been proven for Hamiltonian homeomorphisms of surfaces by Le Calvez \cite{lecalvez05, lecalvez06}.  We have not been able to construct a counterexample to the Conley conjecture in higher dimensions.  
  
  The second question relates to the theory of spectral invariants.  For the sake of simplicity, we limit this discussion to the case of symplectically aspherical manifolds. In that case, the theory of spectral invariants, which was introduced by Viterbo, Oh and Schwarz \cite{viterbo, Oh05b, schwarz}, associates to each smooth Hamiltonian $H$,  a collection of real numbers $\{c(a, H) \in \mathbb{R}: a \in H_*(M) \setminus \{0\} \},$ where $H_*(M)$ denotes the singular homology of $M$.  These numbers are referred to as the spectral invariants of $H$ and they correspond to critical values of the associated action functional.  Hence, the number of distinct spectral invariants of a Hamiltonian $H$ gives a lower bound for the number of fixed points of the time--$1$ map $\phi_H^1$.

    Recall that the cup length of $M$ is defined by $cl(M):= \max \{k+1: \exists \, a_1, \cdots, a_k \in H^*(M)\,:\, \forall i, \deg(a_i)\neq 0 \text{ and }  a_1 \cup \cdots \cup a_k \neq 0\}$. Combining techniques from Hamiltonian Floer theory and Lusternik-Shnirelman theory, Floer \cite{floer89b} and Hofer \cite{hofer} proved that if a Hamiltonian diffeomorphism, of an aspherical symplectic manifold $M$, has fewer spectral invariants than the cup length of $M$, then it must have infinitely many fixed points; see also \cite{howard}.

It is well-known that one can associate spectral invariants to any continuous  Hamiltonian function; see for example \cite{muller-oh}.  In an interesting twist, it turns out that the Hamiltonian homeomorphism that we construct in the proof of Theorem \ref{theo:main}  is \emph{generated by a continuous Hamiltonian which has at least as many distinct spectral invariants as $cl(M)$.} Hence, we see that the correspondence between spectral invariants and fixed points breaks down in the continuous setting. See Remark \ref{rem:cl}.  This leads us to the following question:
  \begin{question}
 Suppose that $H$ is a continuous Hamiltonian with fewer spectral invariants than the cup length of $M$.  Does $\phi^1_H$, the time--1 map of the flow of $H$, have infinitely many fixed points?
  \end{question}  
 
  A positive answer to this question could be interpreted as a $C^0$ version of the Arnold conjecture.

We end this section with a brief discussion which will add to the importance of the above question.  This concerns the theory of barcodes, or persistence modules.  As pointed out in \cite{PS14}, Hamiltonian Floer theory allows one to associate a so-called barcode to any smooth Hamiltonian; see also \cite{Barann,LNV,UZ}.  Barcodes can be viewed as generalizations of spectral invariants. The barcode of a smooth Hamiltonian encodes all the information contained in the filtered Floer homology of that Hamiltonian.  In the same way that one can associate spectral invariants to a continuous function, one can also associate a barcode to a continuous  Hamiltonian function.  In yet another interesting twist, it turns out that the
Hamiltonian homeomorphism of Theorem \ref{theo:main} can be generated by a continuous Hamiltonian which has the same barcode as a $C^2$-small Morse function.  See Remark \ref{rem:barcode}.

\subsection{A brief outline of the construction}
  Construction of the homeomorphism $f$, as prescribed in Theorem \ref{theo:main}, takes place in two major steps.  The first step, which is the more difficult of the two, can be summarized in the following theorem.
  \begin{theo}\label{theo:invariant_tre}
  Let $(M, \omega)$ denote a closed and connected symplectic manifold of dimension at least $4$.  There exists $\psi \in \Hameo(M, \omega)$ and an embedded tree $T \subset M$ such that
  \begin{enumerate}
  \item $T$ is invariant under $\psi$, i.e. $\psi(T) = T$,
  \item All of the fixed points of $\psi$ are contained in $T$, 
  \item $\psi$ is smooth in the complement of $T$.
  \end{enumerate}
  \end{theo}

 For the proof of Theorem \ref{theo:main}, we will in fact  need a refined version of the above result; see Theorem \ref{theo:invariant_tree_2}.
  The proof of this theorem forms the technical heart of our paper.  An important ingredient used in the construction of the invariant tree $T$ is a quantitative $h$-principle for curves.  Quantitative $h$-principles have recently been introduced to $C^0$ symplectic topology by Buhovsky and Opshtein and have had numerous fascinating applications; see \cite{BuOp}.   We should point out that $M$ having dimension at least four is used in a crucial way in the proof of this theorem.  In fact, the second step of the construction, which is outlined below, can be carried out on surfaces as well.
  
  The second major step of our construction consists of ``collapsing''  the invariant tree $T$ to a single point which will be the fixed point of our homeomorphism $f$. Here is a brief outline of how this is done.  Fix a point $p \in M$.  We construct a sequence $\phi_i \in \Symp(M, \omega)$ such that $\phi_i$ converges uniformly to a map $\phi: M \rightarrow M$ with the following two properties:
  \begin{enumerate}
  \item $\phi(T) = {p}$,
  \item $\phi$  is a symplectic diffeomorphism from $M \setminus{T}$ to $ M \setminus\{p\}.$ 
  \end{enumerate}
  Note that the first property implies that $\phi$ is not a 1-1 map and hence, the sequence $\phi_i^{-1}$ is not convergent. Define $f: M \rightarrow M$ as follows: $f(p) =p$ and  $$\forall x \in M\setminus\{p\},\ f(x)=\phi\circ\psi\circ\phi^{-1}(x). $$  It is not difficult to see that $p$ is the unique fixed point of $f$.  Indeed, on $M\setminus\{p\}$, the map $f$ is conjugate to $\psi: M \setminus{T} \rightarrow M\setminus{T}$ which is fixed point free by construction.
 
  By picking the above sequence of symplectomorphisms $\phi_i$ carefully, it is possible to ensure that the sequence of conjugations $\phi_i\psi \phi_i^{-1}$ converges uniformly to $f$.  The uniform convergence of $\phi_i\psi \phi_i^{-1}$ to $f$ relies heavily on the invariance of the tree $T$ and it occurs despite the fact that the sequence $\phi_i^{-1}$ diverges. The details of this are carried out in Section \ref{sec:collapsing_tree}.  It follows that $f$ can be written as the uniform limit of a sequence of Hamiltonian diffeomorphisms.
  
 It is not difficult to see  that $f$ is smooth on the complement of its unique fixed point.  However, proving that $f$ is a Hamiltonian homeomorphism and that it satisfies the first property listed in Theorem \ref{theo:main} requires some more work; see Section \ref{sec:collapsing_tree}. 
  
\subsection{Organization of the paper}  In Section \ref{sec:prelim}, we recall some preliminary results from $C^0$ symplectic geometry.  Symplectic and Hamiltonian homeomorphisms are introduced in Section \ref{sec:sympl-hamilt-home}.  In Section \ref{sec:quant_h_prin}, we introduce a quantitative $h$-principle for curves which plays an important role in our construction.

In Section \ref{sec:collapsing_tree}, we prove that the existence of a Hamiltonian homeomorphism with an invariant tree, as described in Theorem \ref{theo:invariant_tre}, implies the main theorem of the paper.  In Section \ref{sec:invariant_tree}, we prove the existence of a Hamiltonian homeomorphism as described in Theorem \ref{theo:invariant_tre}, assuming a technical and important result: Theorem \ref{theo:invariant_curve}.  Section \ref{sec:invariant_curve}, which occupies the rest of the paper, is dedicated to the proof of Theorem \ref{theo:invariant_curve}.  This section contains the technical heart of the paper. 

\subsection{Acknowledgments}
We would like to thank Yasha Eliashberg, Viktor Ginzburg, Helmut Hofer, R\'emi Leclercq, Fr\'ed\'eric Le Roux,  Patrice Le Calvez, Emmanuel Opshtein, Leonid Polterovich, and Claude Viterbo for fruitful discussions.

  LB: The research leading to this project began while I was a Professeur Invit\'e at the Universit\'e Pierre et Marie Curie. I  would like to express my deep gratitude to the members of the Institut Math\'ematique de Jussieu - Paris Rive Gauche, especially the team Analyse Alg\'ebrique, for their warm hospitality.  
  
  SS:   I would like to thank the members of the Department of Mathematics at MIT, where a part of this project was carried out, for providing a stimulating research environment.

LB was partially supported by the Israel Science Foundation grant 1380/13, by the Alon Fellowship, and by the Raymond and Beverly Sackler Career Development Chair. VH was partially supported by the Agence Nationale de la Recherche, projets ANR-11-JS01-010-01 and ANR-12-BS020-0020. SS was partially supported by the NSF Postdoctoral Fellowship Grant No. DMS-1401569.

\section{Preliminaries from $C^0$-symplectic topology}\label{sec:prelim}
  In this section we introduce some of our notation and recall some of the basic notions of $C^0$--symplectic geometry.  In Section \ref{sec:sympl-hamilt-home}  we give precise definitions for symplectic and Hamiltonian homeomorphisms.  In Section \ref{sec:quant_h_prin} we state a \emph{quantitative h-principle} for curves  which will play a crucial role in the proof of Theorem \ref{theo:main}.

\subsection{Symplectic and Hamiltonian homeomorphisms}
\label{sec:sympl-hamilt-home}
  Throughout the rest of this paper, $(M, \omega)$ will denote a closed and connected symplectic manifold whose dimension is at least $4$.  We equip $M$ with a Riemannian distance $d$. Given two maps $\phi, \psi \co M \to M,$ we   denote
$$d_{C^0}(\phi,\psi)=\sup_{x\in M}d(\phi(x),\psi(x)).$$
We will say that a sequence of maps $\phi_i : M \rightarrow M$, converges uniformly, or $C^0$--converges, to $\phi$, if $d_{C^0}(\phi_i, \phi) \to 0$ as $ i \to \infty$. Of course, the notion of $C^0$--convergence does not depend on the choice of the Riemannian metric.

Recall that a symplectic diffeomorphism is a diffeomorphism $\theta: M \to M$ such that $\theta^* \omega = \omega$. The set of all symplectic diffeomorphisms of $M$ is denoted by $\Symp(M, \omega)$.

\begin{definition} \label{def:sympeo}
	 A homeomorphism  $\theta \co M \to M$  is said to be symplectic if it is the $C^0$--limit of a sequence of symplectic diffeomorphisms.  We will denote the set of all symplectic homeomorphisms by $\Sympeo(M, \omega)$.  
\end{definition}

The Eliashberg--Gromov theorem states that a symplectic homeomorphism which is smooth is itself a symplectic diffeomorphism. We remark that if $\theta$ is a symplectic homeomorphism, then so is $\theta^{-1}$.  In fact, it is easy to see that $\Sympeo(M, \omega)$ forms a group. 

\begin{remark}\label{rem:gen-def} More generally, one can define a symplectic homeomorphism to be a homeomorphism which is locally a $C^0$--limit of symplectic diffeomorphisms; see \cite{BuOp} for further details. 
\end{remark}

Recall that a smooth Hamiltonian $H: [0,1] \times M \to \R$ gives rise to a Hamiltonian flow $\phi^t_H$.  A Hamiltonian diffeomorphism is a diffeomorphism which arises as the time-one map of a Hamiltonian flow.  The set of all Hamiltonian diffeomorphisms is denoted by $\Ham(M, \omega)$; this is a normal subgroup of $\Symp(M, \omega)$.  We next define Hamiltonian homeomorphisms  as introduced by Müller and Oh \cite{muller-oh}.

\begin{definition}[Hamiltonian homeomorphisms]
  \label{def:hameo}
Denote by $B$ an open (possibly not proper) subset of $M$.  Let $(\phi^t)_{t\in [0,1]}$ be an isotopy of $M$ which is compactly supported in $B$.  We say that $\phi^t$ is a hameotopy, or a continuous Hamiltonian flow, of $B$ if there exists a sequence of smooth and compactly supported  Hamiltonians $H_i:[0,1] \times B \to \R$ such that:
\begin{enumerate}
 \item The sequence of flows $\phi_{H_i}^t$ $C^0$--converges to $\phi^t$, uniformly in $t$, i.e. $\displaystyle \max_{t \in [0,1]} d_{C^0}(\phi^t_{H_i}, \phi^t) \to 0$ as $i \to \infty$.
 \item  The sequence of Hamiltonians $H_i$ converges uniformly to a continuous function $H: [0,1] \times M \to \R$, i.e. $\| H_i - H \|_{\infty} \to 0$ as $ i \to \infty$, where $\| \cdot \|_{\infty}$ denotes the sup norm.  Furthermore, 
\end{enumerate}
We say that $H$ generates $\phi^t$, denote $\phi^t=\phi_H^t$, and call $H$ a continuous Hamiltonian. 

A homeomorphism is called a Hamiltonian homeomorphism if it is the time--$1$ map of a continuous Hamiltonian flow.  We will denote the set of all Hamiltonian homeomorphisms by $\Hameo(B, \omega)$.
\end{definition}

It is not difficult to check that  $\Hameo(M, \omega)$ is a normal subgroup of $\Sympeo(M, \omega)$.

A continuous Hamiltonian $H$ generates a unique continuous Hamiltonian flow; see \cite{muller-oh}. Conversely, Viterbo \cite{viterbo06} and Buhovsky--Seyfaddini \cite{buhovsky-seyfaddini} (see also \cite{HLS12}) proved that a continuous Hamiltonian flow has a unique (up to addition of a function of time) continuous generator.

One can easily check that generators of continuous Hamiltonian flows satisfy the same composition formulas as their smooth counterparts. Namely, if $\phi_H^t$ is a continuous Hamiltonian flow, then $(\phi_H^t)^{-1}$ is a continuous Hamiltonian flow generated by $\bar{H}(t,x) = -H(t,\phi^t_H(x))$; given another continuous Hamiltonian flow $\phi_K^t$, the isotopy $\phi_H^t \phi_K^t$ is also a continuous Hamiltonian flow, generated by $H \# K(t,x) := H(t,x)+K(t,(\phi^t_H)^{-1}(x))$.

 \medskip

We will finish this section by recalling an important dynamical property of continuous Hamiltonian flows.  Recall that a submanifold $C$ of a symplectic manifold $(M, \omega)$ is called coisotropic if for all $p\in C$, $(T_pC)^\omega\subset T_pC$ where $(T_pC)^\omega$ denotes the symplectic orthogonal of $T_pC$. For instance, hypersurfaces and Lagrangians are coisotropic. A coisotropic submanifold carries a natural foliation $\m F$ which integrates the distribution $(TC)^\omega$; $\m F$ is called the characteristic foliation of $C$. 

 Assume that $C$ is a closed and connected coisotropic submanifold of $M$ and suppose that $H$ is a smooth Hamiltonian. The following is a standard and important fact  which relates Hamiltonian flows to coisotropic submanifolds:  $H|_C$ is a function of time if and only if $\phi^t_H$ (preserves $C$ and) flows along the characteristic foliation of $C$. By flowing along characteristics we mean that for any point $p\in C$ and any time $t\geq 0$, $\phi_H^t(p)\in\m F(p)$, where $\m F(p)$ stands for the characteristic leaf through $p$.

The following theorem, which was proven in \cite{HLS13}, establishes the aforementioned property for continuous Hamiltonian flows.

\begin{theo}\label{theo:coiso-unique}
  Denote by $C$ a closed and connected coisotropic submanifold and suppose that $\phi^t_H$ is a continuous Hamiltonian flow. The restriction of $H$ to $C$ is a function of time if and only if $\phi^t_H$ preserves $C$ and flows along the leaves of its characteristic foliation.
\end{theo}

 The above theorem indicates that continuous  Hamiltonian flows inherit some of the fundamental dynamical properties of their smooth counterparts. In light of this, it would some reasonable to expect the Arnold conjecture to hold for Hamiltonian homeomorphisms.  But of course, Theorem \ref{theo:main} tells us that this  is quite far from reality.

\subsection{A quantitative $h$-principle for curves} \label{sec:quant_h_prin}
Quantitative $h$--principles were introduced in \cite{BuOp}, where they were used to construct interesting examples of symplectic homeomorphisms.  We will need the following quantitative $h$--principle for curves in the construction of our counterexample to the Arnold conjecture.  
\begin{prop}[Quantitative $h$-principle for curves] \label{prop:h_principle}
Denote by $(W, \omega)$ a symplectic manifold of dimension at least 4. Let $\eps >0$. Suppose that $\gamma_0, \gamma_1 :[0,1] \rightarrow W$ are two smooth embedded curves such that
\begin{enumerate}
\item[i.]$\gamma_0$ and $\gamma_1$  coincide near $t=0$ and $t=1$,
\item[ii.] there exists a homotopy, rel.end points, from $\gamma_0$ to $\gamma_1$  under which the trajectory of any point of $\gamma_0$ has diameter less than $ \eps $, and the symplectic area of the element of $ \pi_2(M,\gamma_1 \sharp \overline{\gamma_0}) $ defined by this homotopy has area $ 0 $.  
\end{enumerate}

  Then, for any $\rho >0$, there exists a compactly supported Hamiltonian $F$, generating a Hamiltonian isotopy $\varphi^s: W \rightarrow W, \; s \in [0,1]$ such that
  \begin{enumerate}
   \item $ F $ vanishes near $ \gamma_0(0) $ and $ \gamma_0(1) $ (in particular, $\varphi^s$ fixes $\gamma_0$ and $\gamma_1$ near the extremities),
   \item  $\varphi^1 \circ \gamma_0 = \gamma_1$,
   \item  $d_{C^0}(\varphi^s, Id) < 2\eps$ for each $s \in [0,1],$ and $\|F\|_{\infty} \leqslant \rho$,
    \item $F$ is supported in a $ 2 \eps$-neighborhood of the image of $\gamma_0$.
  \end{enumerate}
 \end{prop}

The existence of a Hamiltonian $F$ satisfying only properties 1 and 2 is well known. The aspect of the above proposition which is non-standard is the fact that $F$ can be picked such that properties 3 and 4 are satisfied as well. We should point out that the above proposition is a variation of a quantitative $h$--principle for discs which appeared in Theorem 2 of \cite{BuOp}. The proof we will present is an adaptation of the arguments therein. In fact, the quantitative $h$--principle of \cite{BuOp} is considerably more difficult to prove than our proposition and for this reason we will only sketch an overview of the proof of the above proposition.

%
%
%

\begin{proof}[Proof of Proposition \ref{prop:h_principle}]



  First, by a slight Hamiltonian perturbation of $ \gamma_0 $ via a Hamiltonian diffeomorphism generated by a $ C^\infty$-small Hamiltonian function which vanishes near $ \{ \gamma_0(0), \gamma_0(1) \} $, we can, without loss of generality, assume that $ \gamma_0 = \gamma_1 $ on $ [0,\delta] \cup [1-\delta,1] $, and that the images of $ \gamma_0|_{(\delta,1-\delta)} $ and $ \gamma_1|_{(\delta,1-\delta)} $ are disjoint, where $ \delta > 0 $ is a small positive real number. By assumption there exists a homotopy $h(s,t):[0,1] \times [0,1] \rightarrow W$  such that $h(0,t) = \gamma_0(t)$, $h(1,t) = \gamma_1(t)$ and for any fixed $t$, the path $s \mapsto h(s,t)$ is of diameter smaller than $\eps.$  Since the dimension of $W$ is at least $4$, by the weak Whitney immersion theorem, we can  approximate $ h $ by a smooth map $ h' : [0,1] \times [0,1] \rightarrow M $ such that $ h'(0,t) = \gamma_0(t) $ and $ h'(1,t) = \gamma_1(t) $ for $ t \in [0,1] $, $ h'(s,t) = \gamma_0(t) = \gamma_1(t) $ for $ (s,t) \in [0,1] \times ([0,\delta] \cup [1-\delta,1]) $, and such that the restriction $ h'|_{[0,1] \times (\delta,1-\delta)} $ is a smooth immersion with a finite number of self-intersection points occuring inside the relative interior $ h((0,1) \times (\delta, 1-\delta)) $, and whose image $ h'([0,1] \times (\delta,1-\delta)) $ does not intersect $ \gamma_0([0,\delta] \cup [1-\delta,1]) $. Furthermore, similarly as was done in Lemma A.1 from \cite{BuOp}, one can find a smooth map $ h'' : [0,1] \times [0,1] \rightarrow M $ whose image lies in an arbitrarily small neighbourhood of $ h' ([0,1] \times [0,1]) $, such that as before we have $ h''(0,t) = \gamma_0(t) $ and $ h''(1,t) = \gamma_1(t) $ for $ t \in [0,1] $, $ h''(s,t) = \gamma_0(t) = \gamma_1(t) $ for $ (s,t) \in [0,1] \times ([0,\delta] \cup [1-\delta,1]) $, and the image $ h''([0,1] \times (\delta,1-\delta)) $ does not intersect $ \gamma_0([0,\delta] \cup [1-\delta,1]) $, but moreover such that the restriction $ h''|_{[0,1] \times (\delta,1-\delta)} $ is a smooth embedding, and such that for any fixed $t$ the diameter of the curve $s \mapsto h(s,t)$ is less than $2\eps$. Note that by construction, $ h $, $ h' $ and $ h'' $ give the same element of $ \pi_2(M,\gamma_1 \sharp \overline{\gamma_0}) $. Abusing our notation, we will denote $ h'' $ by $ h $ again.
  
  Let $ m $ be a sufficiently large positive integer. Then, for each $ 1 \leqslant i \leqslant m-1 $, the image $h([0,1] \times [\frac{i-1}{m}, \frac{i+1}{m}]) $ has diameter less than $ 2 \eps $, for given $ 0 \leqslant i,j \leqslant m-1 $ we have $h([0,1] \times [\frac{i}{m}, \frac{i+1}{m}])  \cap h([0,1] \times [\frac{j}{m}, \frac{j+1}{m}]) \neq \emptyset $ only if $j \in \{ i-1, i, i+1\}$, and moreover we can find a neighborhood $U_i$ of each $h([0,1] \times [\frac{i}{m}, \frac{i+1}{m}])$ such that $U_i$ is diffeomorphic to a ball,  such that we again have $U_i \cap U_j \neq \emptyset$ only if $j \in \{ i-1, i, i+1\}$, and such that the diameter of $ U_i \cap U_{i+1} $ is less than $2\eps$, for every $ 0 \leqslant i < m $. Moreover, the union $ U  = U_0 \cup \ldots \cup U_m $ can be assumed to be diffeomorphic to a ball, as well as $ U_i \cap U_{i+1} $, for each $ 0 \leqslant i < m $.  Then in particular, $ \omega $ is exact on $ U $, i.e. $ \omega = d \lambda $ on $ U $, for some differential $1$-form $ \lambda $ on $ U $. By our assumptions, $\int_0^1 \gamma_0^* \lambda = \int_0^1 \gamma_1^* \lambda$. 
  

 
 
 \vspace{.2 cm}
 
 \noindent \textbf{Step 1: Mapping points to points.} For each $1 \leq i \leq m$, we pick a Hamiltonian $G_i$ which is supported in $ U_{i-1} \cap U_i $
such that 
$$\phi^1_{G_i}(\gamma_0(t)) = \gamma_1(t), \;\; \forall t  \in [\tfrac{i}{m} - \kappa, \tfrac{i}{m}+ \kappa],$$ 
where $\kappa>0$ is sufficiently small.  In particular, the $G_i$'s have mutually disjoint supports. 
 
 We let $G:= \sum_{i=1}^{m} G_i$ and let $\gamma_0' (t) = \phi^1_{G_i} \circ \gamma_0(t)$. 
We also remark that each $G_i$ can be picked such that $\|G_i\|_{\infty}$ is as small as one wishes. Hence, we may assume that $\|G\|_{\infty} \leq \rho / 3$.

\vspace{.2 cm}

 \noindent \textbf{Step 2: Adjusting the actions.}
  Note that the two curves $\gamma_0|_{[\frac{i}{m}, \frac{i+1}{m}]}$ and  $\gamma_0'|_{[\frac{i}{m}, \frac{i+1}{m}]}$ coincide near their end-points and are both contained in $U_i$.  We would like to find a Hamiltonian diffeomorphism which is supported in $U_i$ and maps $\gamma_0'|_{[\frac{i}{m}, \frac{i+1}{m}]}$ to $\gamma_1|_{[\frac{i}{m}, \frac{i+1}{m}]}$.  However, there is an obstruction to finding such a Hamiltonian diffeomorphism.  These two curves do not necessarily have the same action.   The goal of this step is to modify $\gamma_0'$ to remove this obstruction.

  
  \begin{claim}\label{cl:adjust_action}
  
  There exists a Hamiltonian $H$ with the following properties:
  \begin{enumerate}
  \item The support of $H$ is contained in $ U $, and we have $ \phi_H^t(U_i) \subset U_i $ for every $ t \in [0,1] $ and $ 0 \leqslant i \leqslant m $.
  
  \item The curve $ \gamma_0'' := (\phi_H^1) \circ \gamma_0' $ coincides with $ \gamma_1 $ near $t = \frac{i}{m}$, for $ i =  0,1, \ldots, m $, 
  
  \item $ \gamma_0 ''|_{[\frac{i}{m}, \frac{i+1}{m}]}$ has the same action as $\gamma_1|_{[\frac{i}{m}, \frac{i+1}{m}]}$,  for $ i =  0,1, \ldots, m-1 $,
  
 \item $\|H\|_{\infty} \leq \rho / 3$. 
 
 \item For each $ 0 \leqslant i,j \leqslant m-1 $, the images of $ \gamma_0''|_{(\frac{i}{m},\frac{i+1}{m})} $ and of $ \gamma_0|_{(\frac{j}{m},\frac{j+1}{m})} $ intersect only if $ i = j $.
 
  \end{enumerate}
  \end{claim}
  \begin{proof}
   
 Denote $ \gamma_{0,1}' = \gamma_0' $. We perform $ (m-1) $ steps, where at step $ i $ ($ 1 \leqslant i \leqslant m-1 $) we construct a curve $ \gamma_{0,i+1}' $, and find a Hamiltonian isotopy from $\gamma_{0, i }'$ to $ \gamma_{0, i+1}'$. 
 
 Let us describe the $ i $th step. Let $ \gamma_{0,i}' $ be the curve provided by the previous step. First, perturb the curve $ \gamma_{0,i}' $ in an arbitrarily small neighbourhood of $ t = \frac{i}{m} - \frac{\kappa}{2} $, so that the perturbed curve $ \gamma_{0,i}'' $ satisfies:
 
 \begin{itemize}
 
 \item $ \gamma_{0,i}'' $ coincides with $ \gamma_{0,i}' $ on $ [0, \frac{i}{m} - \frac{3\kappa}{4}] \cup [\frac{i}{m} - \frac{\kappa}{4} ,1] $, 
 
 \item we have $  \gamma_{0,i}''([\frac{i}{m} - \frac{3\kappa}{4} , \frac{i}{m} - \frac{\kappa}{4}]) \subset U_{i-1} \cap U_i $, 
 
 \item the $ \lambda $-actions of 
 the restrictions of $ \gamma_{0,i}'' $ and $ \gamma_1 $ to $ [\frac{i-1}{m},\frac{i}{m}] $ coincide. 
 
\end{itemize} 
 
\noindent Such perturbation can be performed similarly as in the Remark A.13 from \cite{BuOp}. 

Now, we claim that there exists a smooth Hamiltonian function $ H_i $ supported in $ U_{i-1} \cap U_i $, such that $ \gamma_{0,i}'' = (\phi_{H_i}^1) \circ \gamma_{0,i}' $ on $ [0,\frac{i}{m}+\kappa]$, and such that $ \| H_i \|_\infty \leqslant \rho /6 $.  For doing this, roughly speaking, it is sufficient  to isotope (via the Hamiltonian flow) a small segment of $ \gamma_{0,i}' $ so that it coincides with $ \gamma_{0,i}'' $, near $ t = \frac{i}{m} $, or more precisely, for $ t \in [\frac{i}{m} - \frac{3\kappa}{4},  \frac{i}{m} + \kappa] $. And notice, that we are not restricted to keeping the "right end-point" $ \gamma_{0,i}''(\frac{i}{m} + \kappa) $ fixed along the isotopy. Therefore, for keeping the Hofer's norm of the isotopy small, we can just "shrink" the curve  $ \gamma_{0,i}'|_{[\frac{i}{m} - \frac{7\kappa}{8}, \frac{i}{m} + \kappa]} $ to the small segment $ \gamma_{0,i}'|_{[\frac{i}{m} - \frac{7\kappa}{8}, \frac{i}{m} - \frac{3\kappa}{4}]}  = \gamma_{0,i}''|_{[\frac{i}{m} - \frac{7\kappa}{8}, \frac{i}{m} - \frac{3\kappa}{4}]}$ near the left end-point, and then "expand this segment" to coincide with  $ \gamma_{0,i}''|_{[\frac{i}{m} - \frac{7\kappa}{8}, \frac{i}{m} + \kappa] }$.

For a more precise explanation, denote $ a = \frac{i}{m} - \frac{7\kappa}{8} $, $ b = \frac{i}{m} + \kappa $, choose a smooth function $ c : [a,b] \rightarrow [a, a + \frac{\kappa}{8}] $ such that $ c(t) = t $ for $ t \in [a,a+\frac{\kappa}{16}] $ and $ c'(t) > 0 $ on $ [a,b] $, and consider families of curves $ \alpha_s, \beta_s : [a,b] \rightarrow M $, $ s \in [0,1] $, where $ \alpha_s(t) = \gamma_{0,i}'(st + (1-s)c(t)) $, $ \beta_s(t) = \gamma_{0,i}''(st + (1-s)c(t)) $. Note that $ \alpha_0 \equiv \beta_0$. It is easy to see that one can find Hamiltonian functions $ H_i' $, $ H_i'' $, supported in arbitrarily small neighbourhood of $ \gamma_{0,i}'([a+  \frac{\kappa}{8},b]) $ and $ \gamma_{0,i}''([a+  \frac{\kappa}{8},b]) $ respectively, such that $ \alpha_s = \phi_{H_i'}^s \circ \alpha_0 $ and $ \beta_s = \phi_{H_i''}^s \circ \beta_0 $ for each $ s \in [0,1] $, and such that $ \| H_i' \|_{\infty}, \| H_i'' \|_{\infty} \leqslant \rho / 12 $. Now let $ H_i $ be the Hamiltonian function of the Hamiltonian flow $ (\phi_{H_i''}^s \circ (\phi_{H_i'}^s)^{-1})_{s \in [0,1]} $.
The function $ H_i $ is supported in $ U_{i-1} \cap U_i $, satisfies $ \| H_i \|_{\infty} \leqslant \rho / 6 $, and $ \gamma_{0,i}'' = (\phi_{H_i}^1) \circ \gamma_{0,i}' $ on $ [0,\frac{i}{m}+\kappa]$. Now define the curve $ \gamma_{0,i+1}' : [0,1] \rightarrow M $ by $ \gamma_{0,i+1}' := \phi_{H_i}^1 \circ \gamma_{0,i}' $. 

After performing all the $ (m-1) $ steps, the $ \lambda $-actions of $ \gamma_{0,m-1}' $ and $ \gamma_1 $ on $ [\frac{i-1}{m},\frac{i}{m}] $ coincide for any $ 1 \leqslant i \leqslant m-1 $. But since the actions of $ \gamma_{0,m-1}' $ and $ \gamma_1 $ coincide on the whole $ [0,1] $, it also follows that the actions of $ \gamma_{0,m-1}' $ and $ \gamma_1 $ coincide on $ [\frac{m-1}{m},1] $. Note that since all $ H_i $ have disjoint supports (since the support of $ H_i $ is contained in $ U_{i-1} \cap U_i $), if we denote $ H' = H_1 + \ldots + H_{m-1} $, then $ \| H' \|_{\infty} \leqslant \rho / 6 $ and $  \gamma_{0,m-1}'  = (\phi^{1}_{H'}) \circ \gamma_{0,1}' = (\phi^1_{H'}) \circ \gamma_0' $. 

It is possible that for different $ 0 \leqslant i,j \leqslant m-1 $, the images of $ \gamma_{0,m-1}' |_{(\frac{i}{m},\frac{i+1}{m})} $ and of $ \gamma_1|_{(\frac{j}{m},\frac{j+1}{m})} $ intersect. But then, one can easily find a $ C^\infty $-small Hamiltonian function $ H'' $, supported inside an arbitrarily small neighbourhood of $ \bigcup_{i=0}^{m-1} \gamma_1([\frac{i}{m} + \frac{\kappa}{5}, \frac{i+1}{m} - \frac{\kappa}{5}]) $, such that in particular we have $ \| H'' \|_\infty \leqslant \rho / 6 $, such that $ \gamma_0'' := \phi^1_{H''} \circ \gamma_{0,m-1}' $ satisfies the property $ 5 $ from the statement of the Claim, and moreover such that the Hamiltonian function $ H := H'' \sharp H' $ that generates the flow $ (\phi^t_{H''} \circ \phi^t_{H'}) $ satisfies the property $ 1 $ from the statement of the Claim.

\end{proof}
  
  \vspace{.2 cm}
  
\noindent \textbf{ Step 3.  Mapping $ \gamma_0 ''|_{[\frac{i}{m}, \frac{i+1}{m}]}$ to $ \gamma_1|_{[\frac{i}{m}, \frac{i+1}{m}]}$. }

  \begin{claim}\label{cl:nolabel}
  There exist Hamiltonians $K_i$ such that 
  \begin{enumerate}
  \item $K_i$ is supported in $U_i$, 
  \item the support of $K_i$ intersects the images of $\gamma_0''|_{[\frac{j}{m}, \frac{j+1}{m}]}$ and of $\gamma_1|_{[\frac{j}{m}, \frac{j+1}{m}]}$ only for $j = i$,
  \item $\|K_i \| \leq \rho / 6 $,
  \item $\phi^1_{K_i} \circ \gamma_0''|_{[\frac{i}{m}, \frac{i+1}{m}]} = \gamma_1|_{[\frac{i}{m}, \frac{i+1}{m}]}$.
  \end{enumerate}
  \end{claim} 
  \begin{proof}
 
Before passing to the proof, let us remark that in general one can apply directly Lemma A.3 (a) of \cite{BuOp} to our situation, but we would not have obtained the estimate on $\|\cdot\|_\infty$. Therefore the proof is more subtle. Let us roughly explain the steps of the proof. The first step is to make a very small ($ C^\infty $) perturbation of the curve $ \gamma_0'' $ (via the Hamiltonian diffeomorphism $ \phi_{K_i'}^1 $ below), in order to put the curves $ \gamma_0'' |_{[\frac{i}{m}, \frac{i+1}{m}]}$ and $ \gamma_1 |_{[\frac{i}{m}, \frac{i+1}{m}]}$ in general position. Then, we use the following idea. If two curves which coincide near the endpoints and with equal actions, are $ C^\infty $-close, then clearly one can find a very small (with respect to the Hofer's norm) Hamiltonian 
function that moves the first curve to the second. However, if such curves are not $ C^\infty $-close, then we can use a ``conjugation trick'': Instead of moving the first curve to the second via a Hamiltonian diffeomorphism, we find a third curve which is $ C^\infty $-close to the first curve, and such that the pair ``(first curve, second curve)'' could be mapped to the pair ``(first curve, third curve)'' via a symplectomorphism. After that, it is clearly enough to move the first curve to the third curve by a Hamiltonian flow with a very small Hofer's norm, and then conjugate the flow with the symplectomorphism. The details of this are carried out below.
 
 The restrictions $ \gamma_0 ''|_{[\frac{i}{m}, \frac{i+1}{m}]}$ and $ \gamma_1|_{[\frac{i}{m}, \frac{i+1}{m}]}$ both lie in $ U_i $, have the same $ \lambda $-actions, and coincide near the endpoints. Let $ \kappa > 0 $ be such that $ \gamma_0''(t) = \gamma_1(t) $ for $ t \in [\frac{i}{m}, \frac{i}{m} + \kappa] \cup [\frac{i+1}{m}-\kappa,\frac{i+1}{m}] $ (we use the old notation $ \kappa $ in a new situation, and in fact it is enough to replace the old $ \kappa $ by $ \frac{\kappa}{5} $). One can slightly perturb $ \gamma_0'' $ via a Hamiltonian diffeomorphism $ \phi_{K_i'}^1 $ generated by a $ C^\infty $-small Hamiltonian function $ K_i' $, so that $ \gamma_{01}'' :=   \phi_{K_i'}^1 \circ \gamma_0'' $ satisfies  $ \gamma_{01}''(t) = \gamma_1(t) $ for $ t \in [\frac{i}{m}, \frac{i}{m} + \kappa] \cup [\frac{i+1}{m}-\kappa,\frac{i+1}{m}] $, and moreover $ \gamma_{01}''((\frac{i}{m} + \kappa,\frac{i+1}{m}-\kappa)) \cap \gamma_1((\frac{i}{m} + \kappa,\frac{i+1}{m}-\kappa)) = \emptyset $. We may assume that $ K_i' $ is supported in $ U_i $, the support of $K_i' $ intersects the image of $\gamma_0''|_{[\frac{j}{m}, \frac{j+1}{m}]}$ and of $\gamma_1|_{[\frac{j}{m}, \frac{j+1}{m}]}$ only for $j = i$, and that $ \| K_i' \|_{\infty} \leqslant \rho / 12 $.
 
Now, one can clearly find a $ C^\infty $-small Hamiltonian function $ K_i'' $ supported in $ U_i $, so that the support of $K_i'' $ intersects the image of $\gamma_0''|_{[\frac{j}{m}, \frac{j+1}{m}]}$ and of $\gamma_1|_{[\frac{j}{m}, \frac{j+1}{m}]}$ only for $j = i$, and such that the curve $ \gamma_1' := (\phi_{K_i''}^1) \circ \gamma_1 $ coincides with $ \gamma_{01}'' $ on $ [\frac{i}{m}, \frac{i}{m} + \kappa'] \cup [\frac{i+1}{m}-\kappa',\frac{i+1}{m}] $, and moreover $ \gamma_1'((\frac{i}{m} + \kappa',\frac{i+1}{m}-\kappa')) \cap \gamma_1((\frac{i}{m} + \kappa',\frac{i+1}{m}-\kappa')) = \emptyset $, for some $ \kappa' > \kappa $. We may assume that $ \| K_i'' \| \leqslant \rho / 12 $.

Denote $ \kappa'' = (\kappa + \kappa') / 2 $. The curves $ \gamma_{01}''|_{[\frac{i}{m} + \kappa'',\frac{i+1}{m}-\kappa'']} $ and $ \gamma_1'|_{[\frac{i}{m} + \kappa'',\frac{i+1}{m}-\kappa'']} $ lie in $ U_i $, coincide near the endpoints, have the same $ \lambda $-action, and their images do not intersect $ \gamma_1 ([0,1]) $. By Lemma A.3 (a) of \cite{BuOp}, there exists a Hamiltonian function $ K_i''' $ supported in $ U_i $ and away from the endpoints $ \gamma_{01}''(\frac{i}{m} + \kappa'') $, $ \gamma_{01}''(\frac{i+1}{m}-\kappa'') $, such that $ \gamma_1' = \phi_{K_i'''}^1 \circ \gamma_{01}'' $ on $ [\frac{i}{m} + \kappa'',\frac{i+1}{m}-\kappa''] $. By a general position argument and a cut-off argument, we may further assume that the support of $ K_i''' $ does not intersect $ \gamma_1([0,1]) $, as well as $\gamma_0'' ([\frac{j}{m}, \frac{j+1}{m}])$ for $j \neq i$. Denote $ \psi := \phi_{K_i'''}^1 $.

To finish the proof, let $ K_i $ be the Hamiltonian function that generates the flow $ (\psi^{-1} \circ (\phi_{K_i''}^s)^{-1} \circ \psi \circ \phi_{K_i'}^s)_{s \in [0,1]} $. Then $ K_i $ has all the desired properties. For instance, for the property 4 of the statement of the Claim, we first of all have $ \gamma_1 = \psi^{-1} \circ \gamma_1 $ since the support of $ \psi $ does not intersect the curve $ \gamma_1 $, and then we get $$ \gamma_1 = \psi^{-1} \circ \gamma_1 = \psi^{-1} \circ (\phi^1_{K_i''})^{-1} \circ \gamma_1' = \psi^{-1} \circ (\phi^1_{K_i''})^{-1} \circ \psi \circ \gamma_{01}'' $$ $$ = 
\psi^{-1} \circ (\phi^1_{K_i''})^{-1} \circ \psi \circ \phi^1_{K_i'} \circ \gamma_{0}'' =  \phi^1_{K_i} \circ \gamma_0'' $$ on $ [\frac{i}{m}, \frac{i+1}{m}] $.

%

\end{proof}

Let $K_{\mathrm{odd}} = K_1 + K_3 + \cdots $ and $K_{\mathrm{even}} = K_2 + K_4 + \cdots,$ where $K_i$'s are provided to us by the above claim.  We let $K$ be a Hamiltonian such that $\phi^s_K = \phi^s_{K_{\mathrm{odd}}} \circ \phi^s_{K_{\mathrm{even}}}$.  

One can deduce the following facts, without much difficulty, from the above claim: 

\begin{enumerate}
\item $ \phi^s_K(U_i) \subset U_{i-1} \cup U_i \cup U_{i+1} $ for each $ 0 \leqslant i \leqslant m $ and $ s \in [0,1] $ (where we put $ U_{-1} = U_{m+1} = \emptyset $),
\item $\phi^1_K \circ \gamma_0'' = \gamma_1$, 
\item $\|K\|_{\infty} \leq \rho / 3 $.
\end{enumerate}

\medskip

We now let $F$ be a Hamiltonian such that $\phi^s_F = \phi^s_K \circ \phi^s_H \circ \phi^s_G$.  Examining the properties of $K, H,$ and $G$ we see that 

\begin{enumerate}
\item $ F $ vanishes near the extremities of $\gamma_0$ (and hence $\gamma_1$), in particular $\phi^s_F$ fixes the extremities,
\item $\phi^1_F \circ \gamma_0 = \gamma_1$, 
\item $d_{C^0}(\phi^s_F, Id) < 2 \eps$ for each $s \in [0,1]$, and $\|F\|_{\infty} \leq \rho$,
\item $F$ is supported in a $ 2 \eps$--neighborhood of $\gamma_0$.
\end{enumerate}
\end{proof}



\section{Proof of Theorem \ref{theo:main}}

The most important step towards the proof of Theorem \ref{theo:main} will be to establish the following result which is a refined version of Theorem \ref{theo:invariant_tre}.  Throughout this section $(M, \omega)$ will denote a closed and connected symplectic manifold whose dimension is at least four.

\begin{theo}\label{theo:invariant_tree_2}
Let $H$ be a Morse function on $M$.  If $H$ is sufficiently $C^2$--small, then for every $\eps,\rho>0$, there exists $\psi\in\mathrm{Hameo}(M,\omega)$ and an embedded tree $T$ (see Definition \ref{def:tree}) such that:
\begin{itemize}
\item $T$ is $\psi$-invariant, i.e. $\psi(T)=T$,
\item $T$ contains all the fixed points of $\psi$,
\item $d_{C^0}(\phi_H^1,\psi)<\eps$ and $\psi\circ\phi_H^{-1}$ is generated by a continuous Hamiltonian $F$ such that $\|F\|_{\infty}<\rho$,
\item $\psi$ coincides with a Hamiltonian diffeomorphism in the complement of any neighborhood of $T$.
\end{itemize}
\end{theo}

The notion of an \emph{embedded tree} which appears in the above theorem is defined as follows.

\begin{definition}\label{def:tree}
We will say that a compact subset $T$ of a smooth manifold $M$ is \emph{an embedded tree} if there exists a finite tree $T_0$ and an injective continuous map $\chi:T_0\to M$, such that $T=\chi(T_0)$ and the map $\chi$ is a smooth embedding of the interior of each edge of $T_0$. 
\end{definition}

Note that we do not ask more than continuity at the vertices of $T_0$. Note also that the restriction of $\chi$ to any compact interval included in the interior of any edge is a smooth embedding.

In Section \ref{sec:collapsing_tree} below, we explain why Theorem \ref{theo:invariant_tree_2} implies Theorem \ref{theo:main}. The proof of Theorem \ref{theo:invariant_tree_2} then occupies Sections \ref{sec:invariant_tree} and \ref{sec:invariant_curve}.

\subsection{From an invariant tree to a single fixed point}\label{sec:collapsing_tree}

In this section, we explain how one can build a Hamiltonian homeomorphism with a single fixed point from a Hamiltonian homeomorphism preserving an embedded tree (see Definition \ref{def:tree}) that contains all of its fixed points; in other terms, we prove that Theorem \ref{theo:invariant_tree_2} implies Theorem \ref{theo:main}. This will rely on the following proposition.

\begin{prop}\label{prop:from-tree-to-point} Let $\psi$ be a symplectic homeomorphism of $(M,\omega)$, and $T\subset M$ be an embedded tree which is invariant under $\psi$, that is $\psi(T)=T$. Assume that all the fixed points of $\psi$ are contained in $T$. Then, there exists  $f \in \Sympeo(M, \omega)$, with only one fixed point $p$, and such that $f\circ\psi^{-1}\in\Hameo(M,\omega)$. 

Moreover, if $\psi$ is smooth on $M\setminus T$, then $f$ can be chosen to be smooth on $M\setminus\{p\}$. If $\psi$ coincides with a Hamiltonian diffeomorphism in the complement of any neighborhood of $T$, then $f$ can be chosen in any normal subgroup of $\mathrm{Sympeo}(M,\omega)$ which contains $\Ham(M,\omega)$.
\end{prop}

Note that in the last sentence of the above proposition, we do not claim that $f$ can be chosen to be smooth on $M\setminus \{p\}$ and simultaneously be contained in any normal subgroup of $\Sympeo(M, \omega)$ which contains $\Ham(M, \omega)$. As will be clear from the proof, our method of building $f$ in the normal closure of $\Ham(M,\omega)$ has the effect of creating a second non-smooth point. We do not know if both properties can be satisfied at the same time.

\begin{proof}[Proof of Theorem~\ref{theo:main}] Theorem  \ref{theo:invariant_tree_2} provides us with a Hamiltonian homeomorphism $\psi$ which satisfies all the requirements of Proposition 
\ref{prop:from-tree-to-point}. Thus, there exists a symplectic homeomorphism $f$ with only one fixed point $p$ and such that $f\circ\psi^{-1}\in\Hameo(M,\omega)$. Since $\psi\in \Hameo(M,\omega)$ we deduce that $f\in \Hameo(M,\omega)$ as well. Moreover, $\psi$ coincides with a Hamiltonian diffeomorphism in the complement of any neighborhood of $T$. This implies that $f$ can be chosen in any normal subgroup of $\mathrm{Sympeo}(M,\omega)$ containing $\Ham(M,\omega)$. It also implies that $\psi$ is smooth in the complement of $T$ and hence that $f$ can be chosen to be smooth in the complement of $p$.
\end{proof}

The following lemma will be useful for the proof Proposition \ref{prop:from-tree-to-point}.

\begin{lemma}\label{lemma:contractible-tree} Let $(M, \omega)$ be a symplectic manifold of dimension at least 4 and let $T\subset M$ be an embedded tree. Let $d$ be a Riemannian distance on $M$. Then, for every $\eps>0$ and every open neigborhood $U$ of $T$, there exists a Hamiltonian function $H$ supported in $U$ such that $\diam(\phi_H^1(T))\leq \eps$ and $\|H\|_{\infty}\leq \eps$.
\end{lemma}

\begin{proof} 
Let $T_0$ be a finite tree and $\chi:T_0\to T$ a map as in Definition \ref{def:tree}. Let $U$ be an open set containing $T$. We first pick a Hamiltonian function $H_0$ supported in $U$, such that $\phi_{H_0}^1$ maps all the vertices of $T$ inside a ball $B$ included in $U$ and of diameter less than $\eps$. By continuity of $\chi$, there exist an open neighborhood $W$ of the set of vertices of $T$ and a closed subinterval $J_i\subset\mathrm{Int}(I_i)$, for every edge $I_i$, $i=1,\dots,r$ of $T_0$, such that \begin{itemize}
\item $T=(W\cap T)\cup \chi(J_1)\cup\dots\cup \chi(J_r)$,
\item $\phi_{H_0}^1(W)\subset B$.
\end{itemize}
To achieve the proof, we only need to find a Hamiltonian isotopy which moves the pieces of curves $\chi(J_1)$, \ldots , $\chi(J_r)$ into $B$ with the endpoints kept in $B$ along the isotopy. This can be achieved by successively choosing Hamiltonian functions $H_i$'s (for $i=1,\dots,r$) such that for all $i$, $\phi_{H_i}^1(\phi_{H_0}^1(\chi(J_i)))\subset B$  and the support of $H_i$ meets neither    $\phi_{H_0}^1(W\cap T)$ nor  any of the curves $\phi_{H_k}^1(\chi(J_k))$, for $k=1,\dots,i-1$ and $\chi(J_k)$, for $k=i+1,\dots,r$. Then, the Hamiltonian diffeomorphism $h=\phi_{H_r}^1\circ\dots \circ \phi_{H_1}^1\circ \phi_{H_0}^1$ sends $T$ into $B$, hence $\mathrm{diam}(h(T))\leq\eps$.

Moreover, it is a standard fact that the above Hamiltonian functions $H_0, \dots,H_r$ can be chosen arbitrarily small in the $\|\cdot\|_\infty$ norm. Thus, $h$ can be generated by a Hamiltonian $H$ satisfying $\|H\|_{\infty}\leq \eps$.
\end{proof}

\begin{proof}[Proof of Proposition \ref{prop:from-tree-to-point}] Let $\psi$ and $T$ be as in the statement of the proposition. We will build the symplectic homeomorphism $f$ as a $C^0$-limit of conjugates of $\psi$. 

For that purpose, let $W_0\supset W_1\supset\dots\supset W_k\supset\dots\supset T$ be a sequence of nested open neighborhoods of $T$ such that $\bigcap_{k\geq 0}W_k=T$. 
\begin{claim}\label{claim-contracting-sequence}
There exists a sequence of open sets $(U_i)_{i\in\N}$, a sequence of Hamiltonian diffeomorphisms $(h_i=\phi_{H_i}^1)_{i\in\N}$ and a subsequence $(W_{k_i})_{i\in\N}$ with the following properties: For all $i\in\N$, 

\begin{itemize} 
\item $U_i=\varphi_i(W_{k_i})$, where $\varphi_i=h_i\circ\dots\circ h_1\circ h_0$, 
\item $U_{i+1}\subset U_i$,
\item $H_{i+1}$ is supported in $U_i$,
\item $\diam(\overline{U_{i+1}})\leq\frac1{3^i}$,
\item $\|H_i\|_\infty\leq \frac1{3^i}$.
\end{itemize} 
\end{claim}

\begin{proof}
We will construct these sequences by induction. First set $U_0=W_0$, $H_0=0$ and $k_0=0$. Then, assume that we have constructed sequences $(U_i)_{i\in\{0,\ldots,j\}}$, $(H_i)_{i\in\{0,\ldots,j\}}$, $(W_{k_i})_{i\in\{0,\ldots,j\}}$, with the desired properties. 

According to Lemma \ref{lemma:contractible-tree}, applied to the tree $\varphi_j(T)$, we can pick a Hamiltonian function $H_{j+1}$ supported in $U_j$, such that $\|H_{j+1}\|_\infty\leq \frac1{3^{j+1}}$ and, denoting $h_{j+1}=\phi_{H_{j+1}}^1$, $h_{j+1}(\varphi_j(T))$ is included in a ball $B_j\subset U_j$, of diameter less than $\frac1{3^{j}}$. Then pick $k_{j+1}>k_j$ sufficiently large so that 
$$h_{j+1}(\varphi_j(W_{k_{j+1}}))\subset B_j.$$
Then, $\diam(h_{j+1}(\varphi_j(W_{k_{j+1}})))\leq\tfrac1{3^j}$ and we can set $U_{j+1}=h_{j+1}(\varphi_j(W_{k_{j+1}}))$. Since $B_j\subset U_j$, we have $U_{j+1}\subset U_j$ and the three sequences $(U_i)_{i\in\{0,\ldots,j+1\}}$, $(H_i)_{i\in\{0,\ldots,j+1\}}$, $(W_{k_i})_{i\in\{0,\ldots,j+1\}}$ still have the required properties. By induction, we obtain the claimed infinite sequences.
\end{proof}

Since $\diam(\overline{U_i})$ converges to 0  when $i$ goes to infinity and since the $U_i$'s are nested, the intersection of the closures $\bigcap_{i\in\N}\overline{U_i}$ is a single point. Let $p$ denote this point.

Consider the sequence of maps $\varphi_i=h_i\circ\dots\circ h_1$. By construction, if $x\in T$ then  $\varphi_i(x)\in U_i$, hence it converges to $p$. Moreover, for every neighborhood $U$ of $T$, the restrictions  $\varphi_i|_{M\setminus U}$ stabilize for $i$ large. For every $x\notin T$, denote by $\varphi(x)$ the point $\varphi_i(x)$ for $i$ large enough. The map $\varphi$ is a diffeomorphism from $M\setminus T$ to $M\setminus \{p\}$. 

We define for all $x\in M\setminus\{p\}$, $$f(x)=\varphi\circ\psi\circ\varphi^{-1}(x),$$
and $f(p)=p$. 
We see that $p$ is the unique fixed point of $f$. Indeed, if we assume that $f$ admits another fixed point $q\neq p$ then, $\varphi(q)$ would be a fixed point of $\psi$ which is not contained in $T$ and this would be a contradiction.

The first part of Proposition \ref{prop:from-tree-to-point} thus follows if we prove that $f$ is a symplectic homeomorphism. This will be a consequence of the next claim, which requires us to introduce additional notations.

Without loss of generality, we may assume that for all $i$ and all $x$, the Hamiltonian $H_i(t,x)$ vanishes for $t$ in the complement of an open subinterval of $[0,1]$. Let $\tau_0=0$ and for every $i=1,2,\dots$, let $\tau_i=\sum_{k=1}^i\frac1{2^{k}}$. We consider the sequence of smooth Hamiltonian functions $K_i$ defined as concatenations of time-reparametrizations of the $H_i$'s as follows: 
$$K_i(t,x) =\begin{cases} 2^{k+1}H_k(2^{k+1}(t-\tau_k),x), &\forall x\in M,\ k=0,1,\dots,i,\ t\in[\tau_k,\tau_{k+1}]\\
0 & \forall x\in M,\ t\in[\tau_{i+1},1].
\end{cases}
$$
Each Hamiltonian $K_i$ generates the smooth isotopy $\varphi_i^t$ given by:
$$\varphi_i^t =\begin{cases} \phi_{H_k}^{2^{k+1}(t-\tau_k)}\circ h_{k-1}\circ\dots\circ h_0, & \forall k=0,1,\dots,i,\ t\in[\tau_k,\tau_{k+1}]\\
h_i\circ\dots\circ h_0, & \forall t\in[\tau_{i+1},1].
\end{cases}
$$
Note that by construction, for all indices $i<j$, one has $K_i=K_j$ and $\varphi_i^t=\varphi_j^t$ on the time interval $[0,\tau_{i+1}]$. Also, note that for $t=1$, we have $\varphi_i^1=\varphi_i$.

\begin{claim}\label{claim:isotopy}
The isotopies $\varphi_i^t\circ\psi\circ(\varphi_i^t)^{-1}$ $C^0$-converges to an isotopy of homeo\-morphisms $f^t$ as $i\to\infty$, satisfying $f^1=f$. In particular, $f$ is a symplectic homeomorphism. 
\end{claim}

\begin{proof}[Proof of Claim \ref{claim:isotopy}] 
We denote $f_i^t=\varphi_i^t\circ\psi\circ(\varphi_i^t)^{-1}$ and let $f^t$ be the family of homeomorphisms defined for all $x\in M$, all integers $i=0,1,\dots,$ and all $t\in[0,\tau_{i+1}]$ by 
$$f^t(x) = f_i^t,$$
and $f^1(x)=f(x)$. 

First, note that by construction $f$ is a bijection. Note also that on any time interval, of the form $[0,1-\delta]$, with $\delta>0$, $f^t$ coincides with $f_i^t$ for $i$ large enough.
 We will prove that the sequences $f_i^t$ and $(f_i^t)^{-1}$ respectively $C^0$-converge to $f^t$ and $(f^t)^{-1}$. This will imply that $f^t$ is a homeomorphism for all $t$ (including $t=1$) and prove the claim. 

Let $B$ be a ball around $p$. For $i$ large enough, $U_i\subset B$, hence $\varphi_{j}^t(W_{k_i})\subset B$ for all $j\geq i$ and $t\in [\tau_{j+1},1]$. The uniform continuity of $\psi$ and the invariance of $T$ implies that for $j\geq i$ large enough, $\psi(W_{k_j})\subset W_{k_i}$. Thus, 
\begin{equation}\label{eq:fj} f_{j}^t(U_j)=\varphi_{j}^t\circ\psi\circ(\varphi_{j}^t)^{-1}(U_j)\subset B \end{equation} 
 for $j$ large enough and $t\in [\tau_{j+1},1]$. Let $j_0$ be such a large $j$. 

For all $j$, it can be easily checked that  
$$f_{j+1}^t=\rho_j^t\circ f_j^t\circ(\rho_j^t)^{-1},$$
where 
$$ \rho_j^t=\begin{cases}\id, & \forall t\in[0,\tau_{j+1}]\\
\phi_{H_{j+1}}^{2^{j+2}(t-\tau_{j+1})}, & \forall t\in[\tau_{j+1},\tau_{j+2}]\\
h_{j+1}, & \forall t\in[\tau_{j+2},1].
\end{cases}$$
It follows that if $t\in [0,\tau_{j+1}]$, then $(f_{j+1}^t)^{-1}(x)=(f_j^t)^{-1}(x)$ for all $x\in M$. If $t\in[\tau_{j+1}, 1]$, then for all $j\geq j_0$,
  $\rho_j^t$ is supported in $U_{j}\subset U_{j_0}\subset B$, thus for all $x\notin B$, \eqref{eq:fj} yields $(f_j^t)^{-1}\circ (\rho_j^t)^{-1}(x)=(f_j^t)^{-1}(x)\notin U_j$. This implies that $(f_{j+1}^t)^{-1}(x)=(f_j^t)^{-1}(x)$ for all $x\notin B$. Thus, the sequence $(f_i^t)^{-1}$ stabilizes in the complement of $B$, independently of $t$. Moreover, by construction the limit point of $(f_i^t)^{-1}(x)$ for $x\notin B$ is nothing but $(f^t)^{-1}(x)$.

It follows that for $j$ large enough, $f_j^t\circ (f^t)^{-1}$ is supported in $B$, hence is $C^0$-close to $\id$. This shows that $f_i^t$ $C^0$-converges to $f^t$. 
The same argument shows that $(f_j^t)^{-1}$ converges to $(f^t)^{-1}$.
\end{proof}

\medskip
We will now prove that the homeomorphism $f\circ \psi^{-1}$ is a Hamiltonian homeomorphism. 

The homeomorphism $f\circ\psi^{-1}$ is the time one map of the isotopy $f^t\circ \psi^{-1}$. According to Claim \ref{claim:isotopy}, this isotopy is the $C^0$-limit of the isotopies  $\varphi_i^t\circ\psi\circ(\varphi_i^t)^{-1}\circ\psi^{-1}$. Since these isotopies are generated by the continuous Hamiltonians 
$$F_i(t,x)=K_i(t,x)-K_i(t,\varphi_i^t\circ\psi^{-1}\circ(\varphi_i^t)^{-1}(x)),$$ 
the following claim implies that $f^t\circ\psi^{-1}$ is a hameotopy and hence that $f\circ\psi^{-1}$ is a Hamiltonian homeo\-morphism.

\begin{claim}\label{claim:continuous-hamiltonian}The continuous Hamiltonians $F_i$ converge uniformly as $i\to\infty$.
\end{claim}

\begin{proof}
The condition $\|H_i\|_{\infty}\leq \frac1{3^i}$ implies that the sequence of Hamiltonians $K_i$  converges uniformly to the continuous function $K(t,x)$ given for all $x\in M$, all integer $k=0,1,\dots$, and all $t\in[\tau_k,\tau_{k+1}]$ by 
$$K(t,x) = 2^{k+1}H_k(2^{k+1}(t-\tau_k),x),$$
and $K(1,x)=0$.
Now according to Claim \ref{claim:isotopy}, $\varphi_i^t\circ\psi^{-1}\circ(\varphi_i^t)^{-1}$ converges uniformly to $(f^t)^{-1}$. Therefore, $F_i(t,x)$ converges uniformly to $K(t,x)-K(t,(f^t)^{-1}(x))$.
\end{proof}

\medskip
We now pursue the proof of Proposition \ref{prop:from-tree-to-point}. Since $\varphi$ is a smooth diffeomorphism from $M\setminus T$ to $M\setminus\{p\}$, it is obvious that if $\psi$ is smooth in the complement of $T$, then the above construction provides a symplectic homeomorphism $f$ which is smooth in the complement of $p$. 

\medskip
Let us now assume that $\psi$ coincides with a Hamiltonian diffeomorphism on the complement of any neighborhood of $T$. Let $f$ be constructed as above. We will modify $f$ to define a new symplectic homeomorphism $\tilde{f}$ which is in any normal subgroup of $\Sympeo(M,\omega)$ containing $\Ham(M,\omega)$. 

Let $K$ be a smooth Hamiltonian diffeomorphism whose time-one map does not fix the point $p$. Then  $\phi_K^1(p)\neq f\circ\phi_K^1(p)$ and we can find a small enough ball $B$ around $p$, such that $\phi_K^1(B)\cap f\circ\phi_K^1(B)=\emptyset$.

For $i$ large enough, $f$ coincides with $\varphi_i \circ\psi\circ\varphi_i^{-1}$  on $M\setminus B$. Since by assumption $\psi$ coincides with some Hamiltonian diffeomorphism on $M\setminus \varphi_i^{-1}(B)$, we deduce that we can write $f=h\circ g$, where $h$ is a Hamiltonian diffeomorphism and $g$ is a symplectic homeomorphism which is the identity on the complement of $B$. Now set
 $$\tilde{f}=h\circ g\circ\phi_K^{1}\circ g^{-1}\circ (\phi_K^1)^{-1}.$$

We see that $\tilde{f}$ belongs to any normal subgroup of $\mathrm{Sympeo}(M,\omega)$ containing $\Ham(M,\omega)$. We claim that $p$ is the only fixed point of $\tilde{f}$. To see this, first note that $\phi_K^{1}\circ g^{-1}\circ (\phi_K^1)^{-1}$ is the identity on the complement of $\phi_K^1(B)$. It follows in particular that $f=h\circ g$ and $\tilde{f}$ coincide on  $M\setminus \phi_K^1(B)$, hence that $p$ is the only fixed point of $\tilde{f}$ in $M\setminus \phi_K^1(B)$. But if $x\in \phi_K^1(B)$, then $\tilde{f}(x)$ belongs to $f\circ\phi_K^1(B)$, hence is distinct from $x$.

We point out that since $g$ has only one non-smooth point, we see that $\tilde{f}$ has two non-smooth points, $p$ and $\phi_K^1(p)$. 
\end{proof}

\begin{remark}\label{rem:cl} We can now justify what we claimed in the discussion of Section \ref{sec:disc-does-there}, namely the fact that our Hamiltonian homeomorphism $f$ with a unique fixed point can be generated by a continuous Hamiltonian admiting $cl(M)$ distinct spectral invariants. As in Section \ref{sec:disc-does-there}, we restrict our discussion to the case of an aspherical symplectic manifold. 

We begin by recalling that spectral invariants of a smooth Hamiltonian depend on the Hamiltonian Lipschitz continuously; see \cite{schwarz}.  It follows that one can define spectral invariants for any continuous function.

The argument showing that $f$ can be generated by a continuous Hamiltonian with, at least, $cl(M)$ distinct spectral invariants requires four steps:

(i) The initial Hamiltonian $H$ of the construction (see Theorem \ref{theo:invariant_tree_2}) is a $C^2$-small Morse function. For such a Hamiltonian, Floer theory is nothing but Morse theory and it follows from the classical Lusternik-Schnirelman theory that $H$ must have at least $cl(M)$ distinct spectral invariants. 

(ii) The Hamiltonian homeomorphism $\psi$ obtained in Theorem  \ref{theo:invariant_tree_2} can be chosen so that $\psi\circ\phi_H^{-1}$ is generated by a Hamiltonian arbitrarily small in $\|\cdot\|_\infty$ norm. By continuity of spectral invariants, this implies that $\psi$ can be generated by a (continuous) Hamiltonian, which will be denoted by $G$, with spectral invariants close to those of $H$.

(iii) In the above proof of Proposition \ref{prop:from-tree-to-point}, the Hamiltonian homeomorphisms $f_i=\varphi_i\circ\psi\circ\varphi_i^{-1}$ are conjugate to $\psi$.  Hence, they can be generated by the Hamiltonians $G \circ \varphi_i^{-1}$ which have the same spectral invariants as $G$. 

(iv) Our Hamiltonian homeomorphism $f$ is constructed so that for $i$ large, $ f_i^{-1} \circ f $ is generated by a uniformly small Hamiltonian function $ \widetilde{F_i} $.  Thus, $f$ is generated by the Hamiltonian $ (G \circ \varphi_i^{-1}) \# \widetilde{F_i} $ whose spectral invariants are close to those of $G \circ \varphi_i^{-1}$ and hence to those of $H$. 
By choosing small enough perturbations, we can ensure that at least $cl(M)$ of these spectral invariants are distinct. 
\end{remark}

\begin{remark} \label{rem:barcode}
We should point out the argument presented in the above remark can  be modified to prove the following (stronger) statement on closed symplectic manifolds which are not necessarily aspherical:  the Hamiltonian homeomorphism $f$ can be generated by a continuous Hamiltonian, say $G$, whose spectral invariants are exactly the same as the spectral invariants of the initial $C^2$--small Morse function $H$.  

In fact, one could go even further: it is possible to show that the continuous Hamiltonian $G$ has the exact same barcode as the initial $C^2$--small Morse function $H$.  Hence, despite the fact that the time--1 map of $G$ has only one fixed point,  from a Floer theoretic point of view $G$ can not be distinguished from $H$.  For further information on the theory of barcodes see \cite{Barann,LNV,PS14,UZ}.  

Although these claims, and their proofs, are very interesting, in the interest of not lengthening the paper we do not present them here.
\end{remark}

\subsection{Building the tree from a Morse function} \label{sec:invariant_tree}
In this section we begin the proof of Theorem \ref{theo:invariant_tree_2}. Our first step will be to perturb our initial $C^2$-small Morse function $H$ so that it satisfies a number of additional properties. This is the content of the following lemma. In order to simplify our presentation, throughout the rest of this section, we will refer to local maxima/minima of a function as maxima/minima.  An extremum point of a function will be a point which is either a local maximum or a local minimum. 

\begin{lemma}\label{lemma:initial-morse-function}
On every closed symplectic manifold $(M,\omega)$ and for every Morse function $\tilde{H}$ on $M$, there exists a Morse function $H$ on $M$, arbitrarily $C^1$-close to $\tilde{H}$, with the following set of properties:
\begin{enumerate}
\item The function $H$ takes distinct values at distincts critical points,
\item Every critical point $p$ of $H$ which is an extremum admits a neighborhood with Darboux coordinates $(x_1,\ldots,x_n,y_1,\ldots,y_n)$ in which $H$ is of the form $H(p) + c\sum(x_i^2+y_i^2)$, where  $c\in\R\setminus\{0\}$ is a constant which can be chosen to have arbitrarily small magnitude,
\item For every critical point of $H$ which is not extremal, there exist local Darboux coordinates $(x_1,\ldots,x_n,y_1,\ldots,y_n)$ in which $H = c(x_1^2 -y_1^2) +Q$, where $ c $ is some non-zero constant, and $Q$ is a quadratic form in the variables $(x_2, \ldots, x_n, y_2, \ldots, y_n)$.
\end{enumerate}
\end{lemma}

The relevant consequence of the third condition is that locally near $p$, the Hamiltonian flow of $H$ preserves the  symplectic 2-plane $P=\{(x_1,0,\dots,0,y_1, \\ 0,\dots,0)\}$ and acts as a linear hyperbolic flow on it. Indeed, $X_H = X_{ c(x_1^2 - y_1^2)} +  X_{Q}$, where $X_H, X_{c(x_1^2 - y_1^2)},  X_{Q}$ denote the associated Hamiltonian vector fields.   It can easily be checked that  $X_{Q}|_P = 0$.  Hence, the restriction of the  Hamiltonian flow of $H$ to the plane $P$ coincides with the Hamiltonian flow of $ c(x_1^2 -y_1^2) $.

\begin{proof} We can assume without loss of generality that $\tilde H$ is a Morse function on $M$ which has distinct critical values.  We then modify $\tilde H$ near every critical point as follows. Let $p$ be a critical point of $\tilde H$. Up to addition of a constant, we may assume that $\tilde{H}(p)=0$.

Assume that $p$ is a local minimum and let $(x_1,\ldots,x_n, y_1,\ldots,y_n)$ be Darboux coordinates near $p$. We first make a $C^2$-small perturbation of $\tilde H$ so that in a small neighborhood of $p$, it coincides with its Hessian near $p$. Then, note that given two positive definite quadratic forms $Q_1\leq Q_2$ and two open sets $V\supset U\ni 0$, there always exists a smooth function $h$ which coincides with $Q_1$ on $U$ and with $Q_2$ on the complement of $V$, and having 0 as its as only critical point.
Now define $H$ by replacing $\tilde H$ by such a function $h$ obtained in a neighborhood of $p$ from the quadratic forms $Q_1=c\sum x_i^2+y_i^2$ for some small $c>0$ and for $Q_2$ the Hessian of $H$. This perturbation can be made arbitrarily $C^1$-small by using a small neighborhood of $p$. Local maxima are worked out similarly.

Now assume that $p$ is not an extremum. 
By the Morse lemma, and since $p$ is not an extremum, there exists a local chart on a neighbourhood of $ p $, parametrised via coordinates $(v_1,\dots,v_n,w_1,\dots, w_n) $ by a small open ball $ W $ centered at the origin in the Euclidean space $ \mathbb{R}^{2n} $, such that $ p = (0,0, \ldots , 0) $ in these coordinates, and moreover $ \tilde H $ has the form $\tilde{H}= \tilde H(p) + v_1^2-w_1^2+\sum_{i=2}^n\pm v_i^2\pm w_i^2$. Now, let $ (x_1, \ldots, x_n , y_1, \ldots, y_n) $ be Darboux coordinates in a neighbourhood of $ p $, such that $ p = (0,0, \ldots , 0) $ in these coordinates as well. Then, choose a diffeomorphism $ \phi $ of $ M $, supported in a very small ball around $ p $, such that near $p$,  $ \phi $ carries the coordinate system $ (v_1,\dots,v_n,w_1,\dots, w_n) $ to the coordinate system $ (x_1, \ldots, x_n , y_1, \ldots, y_n) $''.
 Consider its ``rescalings'' $ \phi_\lambda: M \rightarrow M $, for $ \lambda \in (0,1) $, defined by $$  \phi_\lambda (v_1,\dots,v_n, w_1,\dots, w_n) = \lambda \phi\left(\frac{v_1}{\lambda},\dots,\frac{v_n}{\lambda},\frac{w_1}{\lambda},\dots, \frac{w_n}{\lambda}\right) ,$$ for $ (v_1,\dots,v_n,w_1,\dots, w_n) \in \lambda W \subset \mathbb{R}^{2n} $, and by the identity on the complement of $ \lambda W $. Now, replacing $ \tilde H $ by the pushforward $ H := (\phi_{\lambda})_* \tilde H = \tilde H \circ (\phi_{\lambda})^{-1} $, we get that for small enough $ \lambda $, $ H $ is a smooth function on $ M $ which is $ C^1 $-close to $ \tilde  H$, and which has the form $ H = H(p) + x_1^2-y_1^2+\sum_{i=2}^n\pm x_i^2\pm y_i^2$ in a small neighbourhood of $ p $.

\end{proof}


In the above lemma, the fact that $H$ and $\tilde{H}$ are $C^1$-close implies that $d_{C^0}(\phi_{\tilde{H}}^1,\phi_H^1)$ is small and $\phi_H^1\circ\phi_{\tilde{H}}^{-1}$ is generated by a uniformly small Hamiltonian. Hence, Theorem \ref{theo:invariant_tree_2} now amounts to the following proposition.

\begin{prop}\label{prop:invariant_tree_2}
Let $H$ be a Morse function on $M$ satisfying the properties of Lemma \ref{lemma:initial-morse-function}.
If $H$ is sufficiently $C^2$-small, then for every $\eps,\rho>0$, there exists $\psi\in\mathrm{Hameo}(M,\omega)$ and an embedded tree $T$ (see Definition \ref{def:tree}) such that:
\begin{itemize}
\item $T$ is $\psi$-invariant, i.e. $\psi(T)=T$,
\item $T$ contains all the fixed points of $\psi$,
\item $d_{C^0}(\phi_H^1,\psi)<\eps$ and $\psi\circ\phi_H^{-1}$ is generated by a continuous Hamiltonian $F$ such that $\|F\|_{\infty}<\rho$,
\item $\psi$ coincides with a Hamiltonian diffeomorphism in the complement of any neighborhood of $T$.
\end{itemize}
\end{prop}

Our proof of Proposition \ref{prop:invariant_tree_2} will make use of the notion of Reeb graph  whose definition we now recall.

\begin{definition}\label{def:connected-component}We assume that a function $H:M\to\R$ is given. For every point $x\in M$ we define $\m C(x)$ as the connected component of $x$ in the level set $H^{-1}(H(x))$. For a subset $X\subset M$, we define $\m C(X)=\bigcup_{x\in X} \m C(x)$.

The Reeb graph is the quotient space $\m R=M/\!\sim $, where $\sim$ is the equivalence relation given by $x\sim y$ if and only if $\m C(x)=\m C(y)$. 
\end{definition}

It follows from basic Morse theory that if $H$ is a Morse function whose critical values are pairwise distinct, then the space $\m R$ actually carries the structure of a graph whose vertices correspond to the critical points of $H$ and such that an edge $e_{p,q}$ between two critical points $p$ and $q$ corresponds to a connected open subset $U_{p,q}\subset M$ such that:
\begin{itemize}
\item $U_{p,q}$ contains no critical point of $H$,
\item The canonical projection $U_{p,q}\to U_{p,q}/\!\sim $ is a fibration onto an interval (this interval can be seen as parametrizing the edge),
\item The closure $\overline{U_{p,q}}$ has two boundary components, one containing $p$ and the other containing $q$.
\end{itemize}

By construction, $H$ descends to a well defined function on $\m R$, which we still denote $H$. This function is monotone on each edge. We assign orientations to the edges so that $H$ is decreasing on each of them. As for any oriented graph, any vertex $p$ admits a bi-degree which is a couple  $(d_-(p),d_+(p))$, where $d_-(p)$ is the number of edges with head end at $p$ and $d_+(p)$ is the number of edges with tail end at $p$. This bi-degree is related to the Morse index of $p$. Indeed, local maxima have bi-degree $(0,1)$, local minima have bi-degree $(1,0)$, and all the other points have bi-degree $(1,1)$, $(2,1)$ or $(1,2)$ (but $(2,1)$ is only possible for points of Morse index $\dim(M)-1$ and $(1,2)$ is only possible for points of Morse index $1$).

\begin{figure}[h!]
\centering
\def\svgwidth{0.6\textwidth}
{\small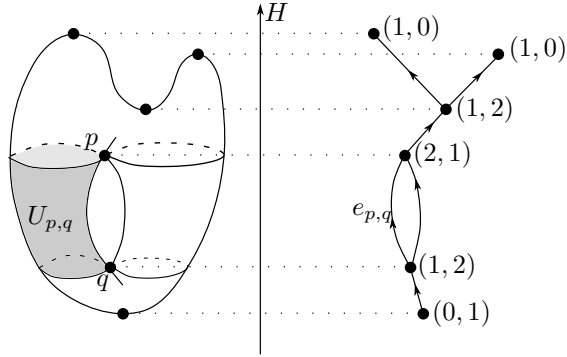}
\caption{The Reeb graph of a height function $H$ on an 2-torus, the bi-degrees of the different vertices (Note that the bi-degree $(1,1)$ cannot happen on a surface) and an example of an edge $e_{p,q}$ with the corresponding open set $U_{p,q}$.}
\label{fig:Reeb_graph}
\end{figure}


\begin{proof}[Proof of Proposition \ref{prop:invariant_tree_2}]
Let $\eps,\rho>0$. Since $M$ is connected, so is the Reeb graph $\m R$ of $H$, and there exists a subgraph $\m T$ of $\m R$ which is a tree and contains all the vertices of $\m R$. To each vertex of $\m T$ we associate the corresponding critical point of $H$. This gives an embedding of the vertices of $\m T$ into $M$. The complicated part of the proof will be to embed the edges. 


By assumption, near every critical point $p$ of $H$ which is not extremal, there exist Darboux coordinates  $(x_1,\ldots,x_n,y_1,\ldots,y_n)$ such that $H= c(x_1^2 -y_1^2) + Q'$, where $c\neq 0$ and $Q'$ is a quadratic form whose kernel is  the plane $P=\{(x_1,0,\dots,0,y_1,0,\dots,0)\}$. As mentioned after the statement of Lemma \ref{lemma:initial-morse-function},    this implies in particular that locally near $p$, the flow of $H$ preserves $P$ and acts on it as the flow of the quadratic form $c(x_1^2-y_1^2)$, which is linear hyperbolic. Thus the flow of $H$ admits in $P$ two orbits converging to $p$ when time goes to $+\infty$ and two orbits converging to $p$ when time goes to $-\infty$. Moreover, in the plane $P$ and still locally near the point $p$, these four orbits are the frontiers of four regions; two of them correspond to $H < H(p)$ and the two others to $H > H(p)$.

To each edge $e_{p,q}$ with tail end at $p$, we associate one of the two orbits converging to $p$ in the past. The local picture above shows that this orbit belongs to the closure $\overline{U_{p,q}}$. Whenever there are two such edges (i.e. $d_+(p)=2$) we demand that the two associated orbits are distinct. Again, the local picture shows that this is possible. We let $x_{p,q}$ be a point located near $p$ on the orbit associated to $e_{p,q}$. To summarize the situation, we have:
$$\overline{U_{p,q}}\ni \phi_H^t(x_{p,q})\stackrel{t\to-\infty}{\longrightarrow} p.$$ 
We associate in a similar way to each edge $e_{q,p}$ with head end at $p$ a point $y_{q,p}$  such that: $$\overline{U_{q,p}}\ni \phi_H^t(y_{q,p})\stackrel{t\to+\infty}{\longrightarrow} p.$$ 

For each oriented edge $e_{p,q}$ we now choose a path $\gamma_{p,q}:[0,1]\to M$ satisfying the following properties:
\begin{itemize}
\item For all $t\in[0,1]$, $\frac{d}{dt}(H\circ \gamma_{p,q}(t))<0$, 
\item $\gamma_{p,q}(0)=p$ if $p$ is a local maximum and $\gamma_{p,q}(0)=x_{p,q}$ otherwise,
\item $\gamma_{p,q}(1)=q$ if $q$ is a local minimum and $\gamma_{p,q}(1)=y_{p,q}$ otherwise.
\end{itemize} 
Note that it follows from the above properties that $U_{p,q}=\m C(\gamma((0,1)))$.

For each non extremal critical point $p$, and for each piece of orbit $\beta$ which is either of the form $\{\phi_H^t(x_{p,q})\,:\,t\in(-\infty,1]\}$ or $\{\phi_H^t(y_{q,p})\,:\,t\in[-1,+\infty)\}$, we choose some open sets $V_p(\beta)$, such that 
$$\mathrm{image}(\beta)\subset V_p(\beta)\subset M\setminus\{p\}.$$   
We choose these open sets sufficiently small such that all the $V_p(\beta)$'s, for $p$ ranging over all non-extremal critical points of $H$ and $\beta$ over all pieces of orbits as above, are pairwise disjoint.

\begin{figure}[h!]
\centering
\def\svgwidth{0.8\textwidth}
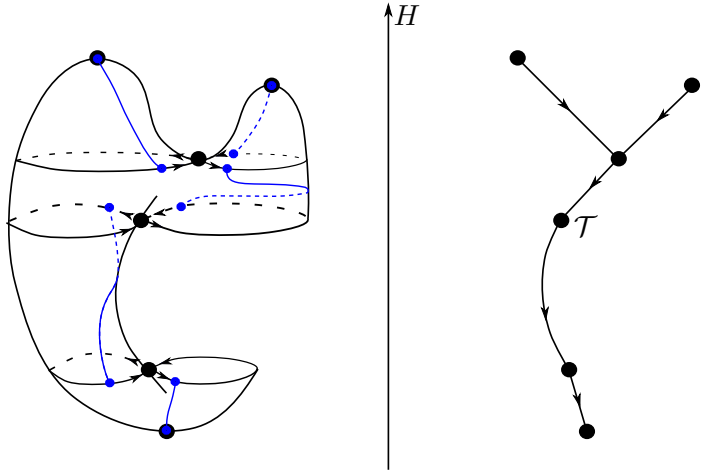
\caption{The right-hand side represents the tree $\mathcal{T}$ obtained from the Reeb graph of Figure \ref{fig:Reeb_graph}. On the left hand side, the arrows correspond to the four orbits converging to non extremal points, and the blue curves represent possible choices for the curves $\gamma_{p,q}$.}
\label{fig:Tree}
\end{figure}

We are now ready to construct the perturbation of $H$ and the embedding of $\m T$. We will proceed  by induction on the edges of $\m T$. For that purpose we number the edges $e_1,e_2,\dots,e_N$, and for each index $i$, we let $p_i$ and $q_i$ be respectively the tail end and head end of $e_i$. With the previous notations, this means that $e_i=e_{p_i,q_i}$. We also rename the points $x_i=x_{p_i,q_i}$ and $y_i=y_{p_i,q_i}$, when they are defined and the paths $\gamma_i=\gamma_{p_i,q_i}$. For the sake of brevity, we omit the piece of orbit from the notation and denote $V_{p_i}$ and $V_{q_i}$ the corresponding open subsets of the form $V_p(\beta)$ when they are defined; it will always be clear from the context which orbit is considered. If $p_i$ is a local maximum, we set $V_{p_i}=\emptyset$; if $q_i$ is a local minimum, we set $V_{q_i}=\emptyset$.

Assume that for some $j\geq 0$, we have realized the edges $e_1,\dots,e_j$ in $M$, i.e, we have built smooth embeddings $\alpha_1,\dots,\alpha_j:\R\to M$ and Hamiltonian homeomorphisms $\theta_1,\dots,\theta_j$, such that the following holds:
\begin{enumerate}
\item $\lim_{t\to-\infty}\alpha_i(t)=p_i$ and $\lim_{t\to+\infty}\alpha_i(t)=q_i$,
\item For all $i\in\{1,\dots,j\}$ and all $t\in\R$, $\theta_i\circ\phi_H^1(\alpha_i(t))=\alpha_i(t+1)$,
\item $\theta_i\circ\phi_H^1$ has the same fixed points as $\phi_H^1$,
\item For all $i\in\{1,\dots,j\}$, $\alpha_i$ takes values in $V_{p_i}\cup \m C(\gamma_i((0,1)))\cup V_{q_i}$, and the images of the $\alpha_i$'s are pairwise disjoint, 
\item For all $i\in\{1,\dots,j\}$, $\theta_i$ is generated by a continuous Hamiltonian $F_i$ supported in $\{p_i\}\cup V_{p_i}\cup \m C(\gamma_i((0,1)))\cup V_{q_i}\cup \{q_i\}$, and the interiors of the supports of the $\theta_i$'s are pairwise disjoint,
\item $d_{C^0}(\theta_i,\id)<\eps$ and $\|F_i\|_{\infty}<\rho$.
\item For any neighborhoods of $p_i,q_i$, the homeomorphism $\theta_i$ coincides with a Hamiltonian diffeomorphism in the complement of the union of those neighborhoods. 
\end{enumerate}

We want to show that it is possible to build $\alpha_{j+1}$ and $\theta_{j+1}$ such that the above properties 1-7 still hold. After shrinking the open sets $V_{p_{j+1}}$, $V_{q_{j+1}}$ if needed so that they intersect neither the supports of $F_1,\dots,F_j$ nor the images of the curves $\alpha_1,\dots,\alpha_j$, we construct $\alpha_{j+1}$ and $\theta_{j+1}$ by applying Theorem \ref{theo:invariant_curve} below to the critical points $p_{j+1}$, $q_{j+1}$, to the curve $\gamma_{j+1}$ and to the open subsets $V_{p_{j+1}}$, $V_{q_{j+1}}$.

Following this process by induction, we get curves $\alpha_1,\dots,\alpha_N$ and Hamiltonian homeomorphisms $\theta_1,\dots,\theta_N$ satisfying Properties 1-7. The union of the images $\alpha_i$, denoted by $T$, is an embedding of the tree $\m T$ in $M$ in the sense of Definition \ref{def:tree}. Moreover, it is easy to check that the Hamiltonian homeomorphism $\psi=\theta_N\circ\dots\circ\theta_1\circ\phi_H^1$ meets the requirements of Proposition \ref{prop:invariant_tree_2}.
\end{proof}

\subsection{Connecting two critical points}\label{sec:invariant_curve}
  The goal of this section is to explain the construction of the individual edges of the invariant tree of Proposition \ref{prop:invariant_tree_2}. This is achieved in Theorem \ref{theo:invariant_curve}.  
  
\subsubsection{Statement of the result}

\begin{theo}\label{theo:invariant_curve}
Let $H$ be a Morse function as described in Lemma \ref{lemma:initial-morse-function} and denote by $p,q$ two critical points of $H$, at most one of them being an extremum. 
We assume that there exists a smooth embedded curve $\gamma:[0,1]\to M$ such that
\begin{itemize}
\item $\gamma(0)=p$ if $p$ is a maximum of $H$, and $\phi_H^t(\gamma(0))\to p$ when $t$ goes to $-\infty$ if $p$ is not a maximum,
\item $\gamma(1)=q$ if $q$ is a minimum of $H$, and $\phi_H^t(\gamma(1))\to q$ when $t$ goes to $+\infty$ if $q$ is not a minimum,
\item for all $t\in[0,1]$, $\frac{d}{dt}(H(\gamma(t))) <0$.
\end{itemize}
Let $\eps,\rho>0$ be positive real numbers,  let $V_p,V_q$ be open sets such that
$$\begin{cases}
V_p =\emptyset, \quad  &\text{ if }p\text{ is a maximum,}\\
V_p \supset \{\phi_H^t(\gamma(0))\,:\,-\infty<t\leq 1\}, \quad&\text{ if }p\text{ is not a maximum,}
\end{cases}
$$
and
$$\begin{cases}
V_q =\emptyset, \quad  &\text{ if }q\text{ is a minimum,}\\
V_q \supset \{\phi_H^t(\gamma(1))\,:\,-1\leq t<+\infty\}, \quad&\text{ if }q\text{ is not a minimum.}
\end{cases}
$$
Then, there exist $\theta \in \Hameo(M,\omega)$ and a smooth embedded curve $\alpha: \R \rightarrow M$, with the following properties:
\begin{enumerate}
\item $\lim_{t \to -\infty} \alpha(t) = p$ and $\lim_{t \to +\infty} \alpha(t) = q$,

\item $\theta \circ\phi^1_H (\alpha(t)) = \alpha(t+1)$ for all $t \in \R$,

\item $\theta \circ \phi^1_H$ has the same set of fixed points as $\phi_H^1$,

\item Fot all $t\in\R$, $\alpha(t)$ belongs to the open set $V_p\cup \m C(\gamma((0,1)))\cup V_q$,

\item $\theta$ is generated by a continuous Hamiltonian $F$ which is supported in 
$\{p\}\cup V_p\cup \m C(\gamma((0,1)))\cup V_q\cup\{q\}$.

\item $d_{C^0}(\id, \theta) < \eps$ and $\|F\|_{\infty}<\rho$. 

\item For any neighborhoods of $p,q$, the homeomorphism $\theta$ coincides with a Hamiltonian diffeomorphism in the complement of the union of those neighborhoods. 

\end{enumerate}
\end{theo}


One of the main difficulties we will have to face in proving Theorem \ref{theo:invariant_curve} is to perform perturbations of $\phi^1_H$ without creating new fixed points. Away from the fixed points of $\phi_H^1$, i.e. critical points of $H$, a $C^0$-small perturbation will not create such fixed points. However, a $C^0$--small perturbation near a critical point can create new fixed points, and for this reason neighborhoods of critical points will require special treatment. In order to surmount these difficulties, we will build the perturbation and the invariant curve of Theorem \ref{theo:invariant_curve} in three steps: First, we build one end of the curve near one of the two critical points (this construction is achieved in Section \ref{section:near-extrema} in the case of a local minimum/maximum). We then extend the invariant curve such that it reaches a sufficiently small neighborhood of the second critical point (Section \ref{section:connecting-2-points}). Finally, we finish the construction of the invariant curve in the neighborhood of the second critical point (Section \ref{section:proof-connecting-2-critical}).

\subsubsection{Connecting a max/min to a nearby point}\label{section:near-extrema}
We denote by $\R^{2n}$ the Euclidean space of dimension $2n$ equipped with the standard symplectic structure $\omega_0 = \sum_{i=1}^{n} dx_i\wedge dy_i$.  In the following theorem, $0$ denotes the origin in $\R^{2n}$, and $\|x\|$ denotes the Euclidean norm of a point $x \in \R^{2n}$.

\begin{theo}\label{theo:invariant_curve_max_to_nearby}
Let $H:  \R^{2n} \rightarrow \R$ be a  Hamiltonian of the form $H =  c \sum_{i=1}^{n}(x_i^2 +y_i^2)$, where $c$ is a non-zero constant.  There exists $A >0$ such that if $|c| < A$, then the following statement holds for any point $x \in \R^{2n}$ and any $\eps, \delta >0$. 

Denote by $B \subset \R^{2n}$ an open ball which is centered at the origin and contains the point $x$.  There exists a Hamiltonian homeomorphism $\theta$, whose support is compactly contained in $B$, and a smooth injective immersion  $\alpha: [-1, \infty) \rightarrow B \setminus\{0\}$  with the following properties:
\begin{enumerate}
\item $\alpha(t) =\phi^t_H(x)$ for $t \in [-1, 0]$ and $\alpha(t) \to 0$ as $t\to \infty$, where $0$ denotes the origin,
\item $\theta \circ \phi^1_H (\alpha(t)) = \alpha(t+1)$ for all $t \in [-1, \infty)$,
\item $\theta \circ \phi^1_H$ has only one fixed point and that is the origin,

\item  The support of $\theta$ intersects the complement of the ball $B(0,\|x\|)$ of radius $\|x\|$ only in an $\eps$-neighborhood of $\{\phi_H^t(x)\,:\,t\in[1,2]\}$,


%
%

\item $d_{C^0}(Id, \theta) < \eps$ and $\|F\|_{\infty} \leq \delta$, where $F$ denotes a continuous Hamiltonian  such that $\phi^1_F = \theta$,

\item  For any neighborhood of $0$, there exists a compactly supported Hamiltonian diffeomorphism of $B$ which coincides with $\theta$ in the complement of that neighborhood.

\end{enumerate}
\end{theo}

In the course of the proof of the above theorem, we will need the following standard fact, which follows from Proposition \ref{prop:h_principle} and hence will not be proven here.  
 
 \begin{lemma} \label{lemma:small_Ham_map_curves}
Denote by $(W, \omega)$ an exact symplectic manifold of dimension at least 4.  Suppose that $\gamma_0, \gamma_1 :[0,1] \rightarrow W$ are two curves such that
\begin{enumerate}
\item[i.] $\gamma_0$ and $\gamma_1$  coincide at $t=0$ and $t=1$,
\item[ii.]  $\int_0^1 \gamma_0^* \lambda = \int_0^1 \gamma_1^* \lambda$ where $\lambda$ is any 1-form such that $\omega = d \lambda$, 
\item[iii.] there exists a homotopy, rel. end points, from $\gamma_0$ to $\gamma_1$.
\end{enumerate}

  Then, for any $\rho >0$, there exists a compactly supported Hamiltonian $F$, generating a Hamiltonian isotopy $\varphi^s: W \rightarrow W, \; s \in [0,1]$ such that
  \begin{enumerate}
   \item  $F$ vanishes near the extremities of $\gamma_0$ and $\gamma_1$, 
   \item  $\varphi^1 \circ \gamma_0 = \gamma_1$,
   \item $\|F\|_{\infty} < \rho$.
  \end{enumerate}
\end{lemma}

\medskip

\begin{proof}[Proof of Theorem \ref{theo:invariant_curve_max_to_nearby}]
Without loss of generality we may assume that $|x| =1$, where $|\cdot|$ denotes the standard Euclidean norm.
  We remark here that throughout the proof we will use the fact that $\phi^t_H$, for each $t$, is a linear isometry of $\R^{2n}$, without explicitly mentioning it.  
We will assume that $\eps >0$ is very small in comparison to $|\phi^1_H(x) - x|$.

We will now pick a sequence of curves $\alpha_0, \alpha_1, \ldots,$ which will be joined together, at a later stage, to form the invariant curve $\alpha$.  Let $x_0 = x$ and define $\alpha_0 :[0,1] \rightarrow B$, $t\mapsto  \phi^t_H(x_0)$.  Let $0 < \rho <1$ be a constant such that $1- \rho$ is very small in comparison to $\eps$.  Now, let $x_1 = \rho\, \phi^1_H(x_0)$ and more generally, for  each $i\geq 1, x_i = \rho^i \phi^i_H(x_0)$ and let $\alpha_i:[i, i+1] \rightarrow B$ denote the curve 
$\alpha_i(t) = \rho^i \phi^t_H(x_0)$ for each $t\in[i,i+1].$ Denote $y_i := \phi^1_H(x_i)$.  Note that $\alpha_i$ satisfies the following identity: 
\begin{equation}\label{eq:alphai}
\alpha_i(t) = \phi^{t-i}_H(x_i) = \phi^{t-i-1}_H(y_i), \; \forall t\in[i,i+1].
\end{equation}

\vspace{.2cm}

\begin{figure}[h!]
\centering
\def\svgwidth{0.6\textwidth}
{\small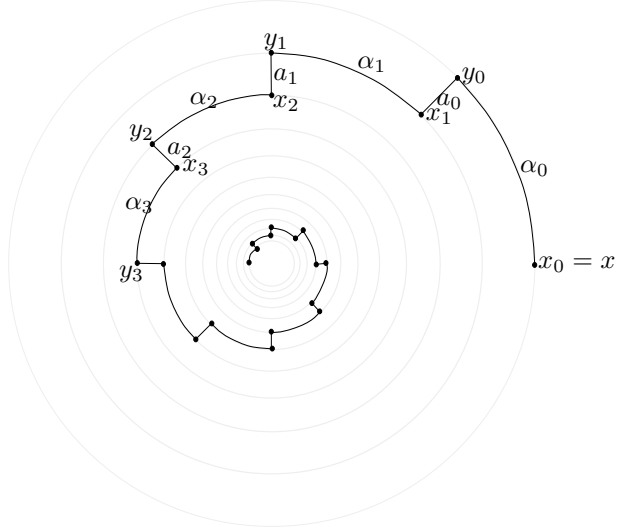}
\caption{Settings for the beginning of the proof of Theorem \ref{theo:invariant_curve_max_to_nearby}.}
\label{fig:Elliptic1}
\end{figure}


\noindent \textbf{Step 1. Preliminary preparations  for the construction of the invariant curve.} 

For each non-negative integer $i$, let $a_i:[0,1] \rightarrow B$  be the curve $a_i(t) := (1-t)y_i + t x_{i+1}$.  Observe that, the length of $a_i$ is $\rho^i - \rho^{i+1}.$ Note that $\rho^i - \rho^{i+1} < \rho^i \eps.$  Furthermore, the image of $a_i$ is disjoint from the remaining $ a_j $'s and is contained in the shell $\{z \in \R^{2n}: \rho^{i+1} \leq |z| \leq \rho^i\}$. See Figure \ref{fig:Elliptic1}.

We will now introduce some of the notation which will be used throughout the proof.  Let $U_0$ be a shell which is a slight enlargement of the shell $\{z \in \R^{2n}: \rho \leq |z| \leq 1\}$.  For $i >0$, we set $U_i:= \rho^i U_0$.  Note that $U_i$ is a shell which is a slight enlargement of  the shell $\{z \in \R^{2n}: \rho^{i+1} \leq |z| \leq \rho^i\}$ and furthermore, $U_i \cap U_j = \emptyset$ if $j \notin \{i-1, i, i+1\}$.  

Next, let $W_0$ be a small neighborhood of the image of the curve $a_0$ which is contained in  $U_0$ and has the following properties:  the diameter of $W_0$ is less than $ \eps$, $W_0$ is convex,  $W_0 \cup \rho W_0$ is also convex, and lastly, there exists a small positive number  $\kappa$ such that  $\phi^t_H(y_0), \phi^t_H(x_{1}) \in W_0$ if and only if $t \in (-2 \kappa, 2 \kappa)$. 
For $i >0$, we define $W_i := \rho^i \phi^i_H(W_0)$, or equivalently, $W_i = \rho \phi^1_H(W_{i-1})$.  What is important to observe about the sets $W_i$ is that they have the following properties:   $W_i$ is a small neighborhood of the image of $a_i$ which is contained in  $U_i$ and the diameter of $W_i$ is less than $ \rho^i \eps$. The sets $W_i$ are convex and so are $\phi^1_H(W_i) \cup  W_{i+1}$. 
  Because $H= c \sum(x_i^2+y_i^2)$ and $|c|\leq A$, for sufficiently small values of $A$, the $W_i$'s are pairwise disjoint and furthermore  $W_j \cap \phi^1_H(W_i) \neq \emptyset$ only if $j=i+1$. Also, note that for $t\in[-1,1]$ 
 we have $\phi^t_H(y_i), \phi^t_H(x_{i+1}) \in W_i$ if and only if $t \in (-2 \kappa, 2 \kappa)$.
Finally, we remark that among the sets $W_j$ the only one which intersects the image of $\phi^1_H\circ \alpha_i$ is $W_i$. Our set up is summarized in Figures \ref{fig:Elliptic1} and \ref{fig:Elliptic2}.


Next, we find Hamiltonians, say $G_i$, such that $G_i$ is supported in $W_i$, and 
\begin{equation}\label{eq:def_Gi} 
\phi^1_{G_i}(\phi_H^t(y_i)) = \phi^t_H(x_{i+1}), \;  \forall \;t \in [-\kappa, \kappa]. 
\end{equation}

 Let $ G: = \sum_{i=0}^{\infty}G_i$; we remark that by Lemma \ref{lemma:small_Ham_map_curves}, we can pick $G_i$ such that $\|G_i\|_{\infty}$ is arbitrarily small.  Note that the $G_i$'s are supported in $W_i$'s which are pairwise disjoint.  Therefore, $G$ is well-defined and in fact $\|G\|_{\infty}$ can be made arbitrarily small. Furthermore, $G$ generates a hameotopy whose flow is the (infinite) composition $\Pi_{i=1}^{\infty} \phi^t_{G_i}$: to see this observe that the flows $\phi^t_{G_i}$ are supported in  $W_i$'s which are disjoint and whose diameters are smaller than $\rho^i \eps$.  We point out that $\phi^1_G$ coincides with a Hamiltonian diffeomorphism of $B$ in the complement of any neighborhood of $0$.
 
 \medskip

Let $\psi_1$ denote the Hamiltonian homeomorphism $\phi^1_G \circ\phi^1_H$.

\vspace{.2cm}
\noindent \textbf{Step 2. A first approximation to  the invariant curve $\alpha$.} 
Consider the smooth curve $\alpha'(t) : [-1,\infty) \rightarrow B$ given by the following formula:  

 $$\alpha'(t):= \begin{cases}\phi_H^{t}(x) & t\in[-1,0], \\
 \phi^1_{G_{i}} \circ \alpha_i(t) \;\; & t\in [i,i+1], \; \forall i \geq 0.\end{cases}
$$

To see that the above formula defines a smooth curve one can check, using Equations \eqref{eq:alphai} and \eqref{eq:def_Gi}, $ \alpha'(t) = \phi^{t-i}_H(x_{i})$  for $t\in[i -\kappa, i + \kappa]$ for each $i\geq1$.  
One can ensure that $\alpha'$ is injective, by perturbing the Hamiltonians $G_i$, if necessary, and by picking the constant $A$, from the statement of the theorem, to be sufficiently small\footnote{ $A \leq \frac{\pi}{2}$ is sufficiently small for our purposes.}; we leave it to the reader to check the details of this.


\begin{figure}[h!]
\centering
\def\svgwidth{0.7\textwidth}
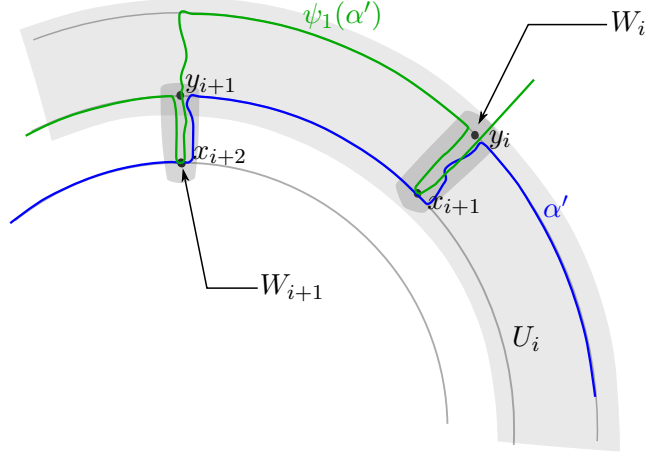
\caption{The curve $\alpha'$ in blue and its image by $\psi_1$ in green.}
\label{fig:Elliptic2}
\end{figure}

  It is evident that $\psi_1( \alpha'(t)) \neq \alpha'(t+1)$ for $t>0$. In Steps 3 and 4, we will be modifying $\alpha'$ and $\psi_1$ to establish the invariance stated in the second property from the statement of the theorem.  The next few claims  record some properties of  these two curves which will be used later in Steps 3 and 4.

\begin{claim}\label{cl:invariance_[-4,0]}
For $A$ sufficiently small, $\psi_1( \alpha'(t)) = \alpha'(t+1)$ for $t\in [-1, 0]$.  
\end{claim}
\begin{proof}

If $t \in[-1,0]$, then one can easily conclude from the definition of $\alpha'$ that we have $\psi_1(\alpha'(t)) = \phi^1_{G_0}\phi^{1}_H \phi^{t}_H(x)= \phi^1_{G_0} \phi^{t+1}_H(x) =  \alpha'(t+1).$ 
\end{proof}

\begin{claim}\label{cl:coincide_endpoints}
For each $i \geq 0$, the two curves $\psi_1 ( \alpha'|_{[i,i+1]})$ and $\alpha'|_{[i+1,i+2]}$ coincide near their endpoints.  More precisely,
$\psi_1 ( \alpha'(t)) = \alpha'(t+1)$ for $t\in [i, i+\kappa] \cup [i+1 - \kappa, i+1]$.
\end{claim}
\begin{proof}
   We will check that if $t\in[i -\kappa, i + \kappa]$, then $\psi_1 \circ \alpha'(t) = \alpha'(t+1)$.  To see this write $t = i+s$, where $s\in [-\kappa, \kappa]$, and note that $\psi_1 \circ \alpha'(t) = \phi^1_G \phi^1_H \phi^{s}_H(x_{i})= \phi^1_G \phi^{s}_H(y_{i}) = \phi^1_{G_{i}} \phi^s_H(y_i) = \phi^s_H(x_{i+1}) = \alpha'(t+1).$ 
\end{proof}

 For each $i \geq 0$, let $V_i$ be the shell  $U_i \cup U_{i+1}$.  Note that $V_i \cap V_j = \emptyset$ if $j \notin \{i-2, i-1, i, i+1, i+2 \}$.

\begin{claim}\label{cl:small_homotopy}
For each $i\geq 0$, there exists a homotopy, which is constant near endpoints,  from $\psi_1 \circ \alpha'|_{[i,i+1]}$ to $\alpha'|_{[i+1,i+2]}$  which is supported in $V_i$ and furthermore, under this homotopy  the trajectory of any point of $\psi_1 \circ \alpha'|_{[i,i+1]}$ has diameter less than $2 \rho^i \eps$.  
\end{claim}
\begin{proof}
We point out that, as a consequence of Claim \ref{cl:coincide_endpoints}, the above two curves coincide near their endpoints.  The idea of the proof of this claim is very simple: the straight-line homotopy $F(s,t) = (1-s) \psi_1 \circ \alpha'|_{[i, i+1]} + s \alpha'|_{[i+1, i+2]}$ satisfies all the required properties.  However, checking the details of this, although quite straight forward, is rather tedious.  Hence, we will omit the proof of this claim. \end{proof}

\noindent \textbf{Step 3. Constructing the curve $\alpha$.}
  
  Let $\lambda$ be the 1-form $\sum_{i=1}^n x_i dy_i$.  We define the action of any  curve $a: [0,1] \rightarrow \R^{2n}$ to be the integral $\int_0^1 a^* \lambda.$  We will now describe how one can make an arbitrarily $C^0$ small perturbation of the curve $\alpha'$ to obtain a new curve $\alpha:[-1, \infty) \rightarrow \R^{2n}$ such that, for each $i\geq -1$, the two curves $\psi_1 \circ \alpha|_{[i,i+1]}$ and $\alpha|_{[i+1,i+2]}$ have the same action.  Furthermore, our curve $\alpha$ will coincide with $\alpha'$ for all values of $t$ except near $t= i+ \frac{1}{2}$, where $i \geq 1$.

   
   We begin by defining $\alpha|_{[-1,1]} = \alpha'|_{[-1,1]}$ and proceed to inductively construct the curve $\alpha$.  Suppose that, for some $k\geq 1$, we have constructed $\alpha:[-1,k] \rightarrow \R^{2n}$ with the properties listed in the previous paragraph.  Next, we make a $C^0$ small perturbation of the curve $\alpha'|_{[k,k+1]}$, in an arbitrarily small neighborhood of the point  $\alpha'(k+ \frac{1}{2})$ to obtain a new curve, which we will call
  $\alpha|_{[k,k+1]}$, such that $\psi_1 \circ \alpha|_{[k-1,k]}$ and $\alpha|_{[k,k+1]}$ have the same action. It is evident from the construction that $\alpha$ coincides with $\alpha'$ for all $t$ except for values of $t$ near $i+ \frac{1}{2}$, where $i\geq 1$.  It is also clear that $\alpha$ be can be picked to be arbitrarily $C^0$ close to $\alpha'$.
  
  We finish this step by pointing out that the curve $\alpha$ satisfies the first of the four properties listed in the statement of Theorem \ref{theo:invariant_curve_max_to_nearby}.
  
\vspace{.2cm}
\noindent \textbf{Step 4.  Turning $\alpha$ into an invariant curve.}  In this final step of the proof, we will perturb $\psi_1$ to a Hamiltonian diffeomorphism $\psi$ such that $\psi \circ \alpha(t) = \alpha(t+1)$.  

\begin{claim}\label{cl:Ki}
For each $i\geq  0$, there exists a Hamiltonian $K_i$ supported in $V_i = U_i\cup U_{i+1}$ such that
  \begin{enumerate}
     \item $\phi_{K_i}^1 \circ \psi_1 \circ\alpha(t) = \alpha(t+1)$ for each $t\in[i,i+1]$, 
     \item $d_{C^0}(Id,  \phi_{K_i}^1) < 4\rho^i \eps$ and  $K_i$  can be picked such that $\|K_i\|_{\infty}$ is as small as one wishes,
     \item $ \mathrm{supp}(K_i)$ is contained within the $4\eps$ neighborhood of $\alpha([i+1,i+2])$,
     \item  $\alpha([j, j+1]) \cap \mathrm{supp}(K_i) =  \emptyset$ if $ j \neq i+1$.  Similarly, $ \psi_1 \circ \alpha([j,j+1]) \cap \mathrm{supp}(K_i)= \emptyset$ for $j\neq i$.
    \end{enumerate}
\end{claim}
\begin{proof}    
In the previous step we constructed $\alpha$ such that $\psi_1\circ\alpha|_{[i,i+1]}$ and $\alpha|_{[i+1, i+2]}$ have the same action.  Furthermore, since we obtained $\alpha|_{[i,i+1]}$ from $\alpha'|_{[i, i+1]}$ by making a $C^0$-small perturbation of $\alpha'|_{[i, i+1]}$ near $t= i+ \frac{1}{2}$ we conclude, by Claim \ref{cl:small_homotopy}, that there exists a homotopy, which is constant near endpoints,  from $\psi_1 \circ \alpha|_{[i,i+1]}$ to $\alpha|_{[i+1,i+2]}$  which is supported in $V_i$ and furthermore, under this homotopy  the trajectory of any point of $\psi_1 \circ \alpha|_{[i,i+1]}$ has diameter less than $2 \rho^i \eps$. The first three properties of the claim then follow immediately from Proposition \ref{prop:h_principle}.  As for the fourth property, since the dimension of $M$ is at least four, by making a small perturbation of the homotopy from $\psi_1 \circ \alpha|_{[i,i+1]}$ to $\alpha|_{[i+1,i+2]}$, we may assume the image of the homotopy does not intersect $\alpha([j, j+1])$, for $ j \neq i+1$, or $\psi_1 \circ \alpha([j,j+1])$, for $j \neq i$.   Fix small $ \kappa_i > 0 $ such that the homotopy is constant on $\psi_1 \circ \alpha|_{[i,i+\kappa_i] \cup [i+1-\kappa_i]} $.  Let $W$ be a small neighborhood of the image of the restricted homotopy from $\psi_1 \circ \alpha|_{[i+\kappa_i,i+1-\kappa_i]}$ to $\alpha|_{[i+1+\kappa_i,i+2-\kappa_i]}$ such that $\alpha([j, j+1]) \cap W =  \emptyset$, if $ j \neq i+1$ and $\psi_1 \circ \alpha([j,j+1]) \cap W = \emptyset$ for $j\neq i$.  Now, it is easy to see that $K_i$ can be picked to have its support contained in $W$.  This implies the fourth property of the claim.
\end{proof}

Note that for each $i$, $V_i \cap V_j \neq \emptyset$ only if $j\in \{i-2, i-1, i, i+1, i+2\}$.  
Hence, we see that the sets $ V_0, V_3, V_6, \ldots$ are mutually disjoint and therefore the supports of $ K_0, K_3, K_6, \ldots$ are mutually disjoint as well.  This combined with the fact that  $\|K_i\|_{\infty}$ can be picked to be as small as we wish, implies that the sum $F_0:= K_0+ K_3 + K_6  + \cdots$ defines a continuous function.  Furthermore, as a consequence of the disjointness of the supports of these functions, we see that $F_0$ is a continuous Hamiltonian whose  flow is 
$\phi^t_{F_0}= \phi^t_{K_0} \circ \phi^t_{K_3} \circ \phi^t_{K_6} \cdots $.  Similarly, we define $F_1 = K_1+K_4 +K_7 + \cdots$ and $F_2 =  K_2 +  K_5 +K_8 +  \cdots$.  These functions generate the hameotopies  $\phi^t_{F_1}=  \phi^t_{K_1} \circ \phi^t_{K_4} \circ \phi^t_{K_7}  \cdots $ and $\phi^t_{F_2}= \phi^t_{K_2} \circ \phi^t_{K_5} \circ \phi^t_{K_8} \cdots $, respectively.    Observe that, as a consequence of the second item in Claim \ref{cl:Ki}, $F_0, F_1, F_2$ can be picked such that their norms are as small as one wishes.

We define $\theta := \phi^1_{F_2}\circ \phi^1_{F_1}\circ \phi^1_{F_0}\circ \phi^1_G$ and $\psi:= \theta \circ \phi^1_H = \phi^1_{F_2} \circ\phi^1_{F_1} \circ\phi^1_{F_0}\circ\psi_1$.  Clearly, $\theta$ is a Hamiltonian homeomorphism which is compactly supported in $B$. 

 We will now check that $\theta$ and the curve $\alpha$ satisfy the properties listed in the statement of Theorem \ref{theo:invariant_curve_max_to_nearby}.  We have already checked the fact that the first property is satisfied.       

First, we will show that $\psi( \alpha(t)) = \alpha(t+1)$, for all $t\in[-1, \infty)$.  If $t \in[-1,0]$, then this follows from Claim \ref{cl:invariance_[-4,0]}.  Indeed, for $t \in[-1,1]$ we have $\alpha(t) = \alpha'(t)$. Moreover, for $ t \in [-1,0] $ we have $\psi(\alpha'(t)) =  \phi^1_{F_2} \circ\phi^1_{F_1} \circ\phi^1_{F_0} (\psi_1( \alpha'(t))) = \phi^1_{F_2} \circ\phi^1_{F_1} \circ\phi^1_{F_0} (\alpha'(t+1)) = \alpha'(t+1)$.  The reason the last equality holds is that, by the 4th item of Claim \ref{cl:Ki}, $\alpha'(t+1) = \alpha(t+1) $ is not contained in the support of any of $F_0, F_1, F_2$.  Next, consider $t \in [0, \infty)$.  Fix $i$ and $t\in[i,i+1]$.  According to Claim \ref{cl:Ki}, $\alpha(t+1), \psi(\alpha(t)) \notin \mathrm{supp}(K_j)$ for any $j$ other than $j=i$.  Hence,  $\psi( \alpha(t)) = \phi^1_{F_2} \circ\phi^1_{F_1}\circ \phi^1_{F_0}\circ \psi_1( \alpha(t)) = \phi^1_{K_i} \circ \psi_1( \alpha(t)) = \alpha(t+1) $. 

 To establish the remaining properties  we will need the following claim.  We define a nested sequence of balls $ B_i \supset B_{i+1}$ as follows: For $i\geq 0$  we set $B_i = \{0\} \cup(\cup_{k\geq i} V_k)$.  Note that $B_0$ is a slight enlargement of the unit ball and furthermore $B_0$ contains the supports of all $F_0, F_1, F_2$, and $G$.  
\begin{claim}\label{cl:everything_C0_small}
Suppose that $p\in B_0 - \{ 0\}$. Let $i$ denote the smallest integer such that $ p \in U_i$.  Then, for each $j\in \{0,1,2\}$, 
\begin{itemize}
\item if $i>0$, then $\phi^1_{F_j}(p) \in B_{i-1}$ and $ |\phi^1_{F_j}(p) - p| \leq 4\rho^{i-1} \eps$ .
\item if $i=0$, then  $\phi^1_{F_j}(p) \in B_{0}$ and $ |\phi^1_{F_j}(p) - p| \leq 4 \eps$.
\item $\phi^1_G(p) \in B_{i-1}$ and $|\phi^1_G(p) -p| \leq \rho^{i-1} \eps$.
\end{itemize}
\end{claim}
\begin{proof} We begin with the statement about $\phi^1_{F_j}(p)$.  We will only prove this for $j=0, i>0$ and leave the remaining cases to the reader. Recall that $F_0 = K_0 + K_3 + K_6 + \cdots$ and that $\mathrm{supp}(K_m) \subset V_m = U_m \cup U_{m+1}$ for all $m$.  Hence, the point $p$ can only be in the support of $K_{i-1}, K_i, K_{i+1}$.  Now, only one of these three Hamiltonians enters the definition of  $F_0$ and so we see that $\phi^1_{F_0}(p) \in \{\phi^1_{K_{i-1}}(p), \phi^1_{K_{i}}(p), \phi^1_{K_{i+1}}(p)\}$.  The result then follows from Claim \ref{cl:Ki}.  The statement about $\phi^1_G(p)$ is proven similarly.
\end{proof}

We will now prove that $\psi$ has no fixed points other than $0$.  Recall that $ \psi = \theta \circ\phi^1_H.$  Using the above claim, and the fact that $\theta = \phi^1_{F_2} \circ\phi^1_{F_1} \circ\phi^1_{F_0} \circ\phi^1_G$, we see that 
\begin{enumerate}
\item $|\theta \circ\phi^1_H(p) -\phi^1_H(p)| \leq  13 \rho^{i-4} \eps$, when $p \in B_i$ and $i\geq 4$
\item $|\theta \circ\phi^1_H(p) -\phi^1_H(p)| \leq 13\eps$, when $p \in B_i$ and $i\leq3.$
\end{enumerate}
 
 Suppose that $p \in U_i$ where $i \geq4$. We will show that $p$ can not be a fixed point of $\theta \circ \phi^1_H$.  First, note that since $U_i$ is a slight enlargement of the shell $\{z: \rho^{i+1} \leq |z| \leq \rho^{i} \}$, we can assume that $|p| \geq \rho^{i+2}.$ Recall that, we picked $\eps$ to be very small in comparison to the number $C= |\phi^1_H(z) -z|$, where $z$ is any point such that $|z|=1$. It follows that $|\phi^1_H(p) -p| \geq C \rho^{i+2}$.  Hence, we see that if $C \rho^{i+2} > 13 \rho^{i-4} \eps$, then $\theta \circ\phi^1_H(p) \neq p$.   Of course, by picking $\eps$ to be sufficiently small we can make sure that the inequality $C \rho^{i+2} > 13 \rho^{i-4} \eps$ holds.  (Recall that we picked $\rho$ such that $1-\rho$ is small in comparison to $\eps$.)  We leave the case where $p\in U_i$ and $i \leq 3$ to the reader.
 
 \medskip 
 
 Next, we will show that the support of $\theta$ satisfies the fourth property from the statement of the theorem.  Indeed, examining the proof we can see that the Hamiltonians $G_i, K_i$ were all picked to be  supported within at most a $4\eps$ neighborhood of $\alpha([i+1, i+2])$.  Of these Hamiltonians it is only $G_0, K_0$ whose supports intersect the complement of $B(0, \|x\|)$.  This implies that the support of $\theta$ intersects the complement of $B(0, \|x\|)$ only in a $4\eps$--neighborhood of  $ \alpha([1,2])$.  Now, we leave it to the reader to check that $ \alpha([1,2])$ is contained within an $\eps$--neighborhood of $\{\phi^t_H(x): t \in [1,2] \}$.
 Replacing $\eps$ by $\frac{\eps}{5}$ throughout the proof yields the result.

  We now check the 5th property from the statement of the theorem.  Since $\theta = \phi^1_{F_2} \circ\phi^1_{F_1} \circ\phi^1_{F_0} \circ\phi^1_G$, it is an immediate consequence of Claim \ref{cl:everything_C0_small} that $d_{C^0}(\theta, Id) \leq 13 \eps$.  Hence, if we replace $\eps$ by $\frac{\eps}{13}$ then we obtain $d_{C^0}(Id, \theta) \leq \eps.$  Next, let $F$ denote the generating Hamiltonian of the continuous Hamiltonian flow $\phi^t_{F_2}\circ \phi^t_{F_1}\circ \phi^t_{F_0}\circ \phi^t_G$.  It follows from the composition formulas mentioned in Section \ref{sec:sympl-hamilt-home} that $\|F\|_{\infty} \leq \|F_0\|_{\infty} + \|F_1\|_{\infty} + \|F_2\|_{\infty} + \| G \|_{\infty}$.  As we have already mentioned, $F_0, F_1, F_2, G$ can be picked to have norms as small one wishes.  Hence, the same is true for $\|F\|_{\infty}$.

   Finally, we check the 6th and final property. Recall the nested sequence of balls $B_i \supset B_{i+1}$ introduced before Claim \ref{cl:everything_C0_small}. Let $B_i^c$ denote the complement of $B_i$. Observe that in $B_i^c$, for any $i$, each of $\phi^1_{F_0}, \phi^1_{F_1}, \phi^1_{F_2}$, and $\phi^1_G$ coincides with a Hamiltonian diffeomorphism.  Furthermore, using Claim \ref{cl:everything_C0_small}, it can easily be checked that each of these homeomorphisms maps $ B_i^c$ into  $B_{i+1}^c$.  Combining these facts together we see that  $\theta$ coincides with a Hamiltonian diffeomorphism in the complement of each $B_i$. 
\end{proof}

\subsubsection{Connecting two non critical points by an invariant curve }\label{section:connecting-2-points}

The main goal of this section is to prove the following technical result which is needed for the construction of the invariant curve of Theorem \ref{theo:invariant_curve}.

\begin{theo}\label{theo:connecting-non-critical}
Denote by $ H: M \rightarrow \R $ a smooth autonomous Hamiltonian on a closed symplectic manifold $(M, \omega)$ of dimension at least 4. Let $ z,w$ be two points in  $M$.  Suppose that there exists a smooth embedded curve $\gamma : [0,1] \rightarrow M $ such that 
 $ \gamma(0) = z,  \gamma(1) = w$,  such that $ H $ is not constant on $ \gamma $,
and  the map $ \Gamma : [0,1] \times \mathbb{R} \rightarrow M $ defined by $ \Gamma(s,t) = \phi_{H}^t(\gamma(s)) $ is a smooth embedding on
 $(\{0\}\times[-2,1])\cup((0,1)\times[0,1]) \cup(\{1\}\times[-1,2])$. 

Then, for any constants $ \eps, \rho > 0 $ and 
 any neighborhood $ U $ of $ \Gamma([0,1] \times [0,1]) $, there exist

\begin{itemize}
\item a Hamiltonian diffeomorphism $ \psi : M \rightarrow M $,
\item and a smooth embedded curve $ \alpha : [0,k+1] \rightarrow M  $, where $k \in \Z$,
\end{itemize}

{\bf such that:} 

\begin{enumerate}

\item $ \psi(\alpha(s)) =   \alpha(s+1) $ for any $ s \in [0,k] $, 

\item There exists $t_0 \in (1,2)$ such that for all $s\in[0,t_0]$, $\alpha(s)=\phi_H^{s-2}(z)$, and $\alpha([t_0, k+1]) \subset\displaystyle U\cup \m C(\gamma((0,1)))$,



\item  Let $V$ be any neighborhood of the point  $\phi_H^{\frac12}(w)$.  Then, $\alpha$ can be chosen to satisfy the following property:
$$\alpha(k+s)\begin{cases}=\phi_H^{s}(w), & s\in[0,\tfrac12-b]\cup[\tfrac12+b,1]\\
\in  V, & s\in[\tfrac12-b,\tfrac12+b],\end{cases}
$$ for some small $b>0$.  Furthermore,  $ \psi \circ \alpha(k+s) = \phi_H^{s+1}(w) $ for $ s \in[0,\tfrac12-b]\cup[\tfrac12+b,1] $,

\item   $ d_{C^0}(\psi, \phi_{H}^1) < \eps $  and the Hamiltonian diffeomorphism $\theta :=  \psi \circ \phi_{H}^{-1} $  is  generated by a Hamiltonian, say $F$, such that $ \| F  \|_{\infty}<\rho$. Furthermore,  $F$ is supported inside $\displaystyle U\cup \m C(\gamma((0,1)))$,

\item $\psi$ has the same set of fixed points as $\phi^1_H$. 
\end{enumerate}
\end{theo}

\begin{proof}[Proof of Theorem \ref{theo:connecting-non-critical}]

By replacing $U$ with a smaller open subset we can assume that there exists a diffeomorphism  $\Phi: U \rightarrow (-c, 1+c) \times (-c, 1+c) \times (-c,c) \times \ldots \times (-c,c) \subset \R^{2n}$, where $c>0$ is sufficiently small. By slightly decreasing $ c $, if necessary, we may assume that $ \Phi $ is a bi-Lipschitz map, where we consider the metric $ g $ on $ M $ and the standard euclidean metric on $ (-c, 1+c) \times (-c, 1+c) \times (-c,c) \times \ldots \times (-c,c) $. We can further suppose that this diffeomorphism identifies $\Gamma(s,t)$ with $(s,t, 0, \ldots, 0)$ for all $s, t \in [0,1]$.  Note that here we are relying on the fact that $\Gamma$ is an embedding on $[0,1] \times [0,1]$. 

\noindent \textbf{Step 1. Preliminary preparations  for the construction of the invariant curve.}  
We begin by picking $\delta >0$ such that  $\delta<\frac{\eps}{2}$ and   $\forall x, y \in M$
\begin{equation}\label{eq:delta}
  \mathrm{if} \; d(x,y) < \delta, \; \mathrm{then} \; \forall t\in [0,1],\; d(\phi^t_H(x), \phi^t_H(y)) < \frac{\eps}{2}.
\end{equation}
Pick $m \in \N$  large enough  such that $d(\gamma(\frac{i}{m}), \gamma(\frac{i+1}{m})) < \frac{\delta}{2}.$ For each $ 0 \leq i < m $, we define small neighborhoods $ U_i $ of  $ \Gamma([\tfrac{i}{m}, \tfrac{i+1}{m}] \times [0,1]) $ by 
   $$U_i := \Phi^{-1} ((\tfrac{i}{m} - a, \tfrac{i+1}{m} +a )  \times (-c, 1+c) \times (-c,c) \times \ldots \times (-c,c)), $$
where $a>0$ is taken to be so small  that $ U_i \cap U_j \neq \emptyset$ only if $j \in\{ i-1, i, i+1\}$. Clearly, $U_i \subset U$ for all $i$.
 
Take $r>0$ to be small in comparison to $\delta$ such that the $ r $--neighborhood of $ \gamma(\frac{i}{m}) $ is contained in $U_i$.  For each $i$ choose a ball $ B_i $ in the $ r $--neighborhood of $ \gamma(\frac{i}{m}) $ such that $ B_i \subset U_i \cap C(\gamma((0,1)))$ and $ H(B_i) \cap H(B_j) = \emptyset $ for any $ i \neq j $. It is not hard to show that such a choice of $B_i$ is always possible; here it is important to remember that, because $\Gamma$ is an embedding,  $ \gamma $ does not pass through any of the critical points of $ H $.  Note that we may also assume that $B_{i+1} \subset U_i$ for $0\leq i < m-1$.

  We  can ensure that for each $x \in B_i \cup B_{i+1}$ the image of the curve $t \mapsto \phi^t_H(x), t\in [0,1] $ is contained in $U_i$; this will be used in Step 4.

 By the Poincar\'e recurrence theorem, we can find points $x_i \in B_i, 0 < i <m$, and integers $k_{i} \geq 2$  with the property that $\phi_H^{k_{i}}(x_i) \in  B_i$.   Let $y_i := \phi_H^{k_{i}}(x_i)$.  By an arbitrarily $ C^2 $-small perturbation of $ H $ inside the open set $ U $ away from the curves $ (\phi_H^{t}(x_0))_{t \in [-2,0]} $ and $ (\phi_H^{t}(x_m))_{t \in [0,2]} $, if needed, we may assume that  the curve $ (\phi_H^t(x_i))_{t \in [-1,k_i+1]} $ is embedded, for each $ i $. 
If  $i$ is $0$ or $m$, we set $x_0 = \phi^{-2}_H(z), y_0 = z, k_0 =2$ and $x_m = w, y_m = \phi^1_H(x_m), k_m =1$. 
 We remark that it is possible to pick the points $x_i$ such that  for each $ 0 \leqslant i < m $, the points $x_i$ and $ x_{i+1} $ do not belong to the same level set of $H$; this will be used in the next step of the proof.  
 Figure \ref{fig:Gamma}, below, describes the settings from Step 1. 

\begin{figure}[h!]
\centering
\def\svgwidth{\textwidth}
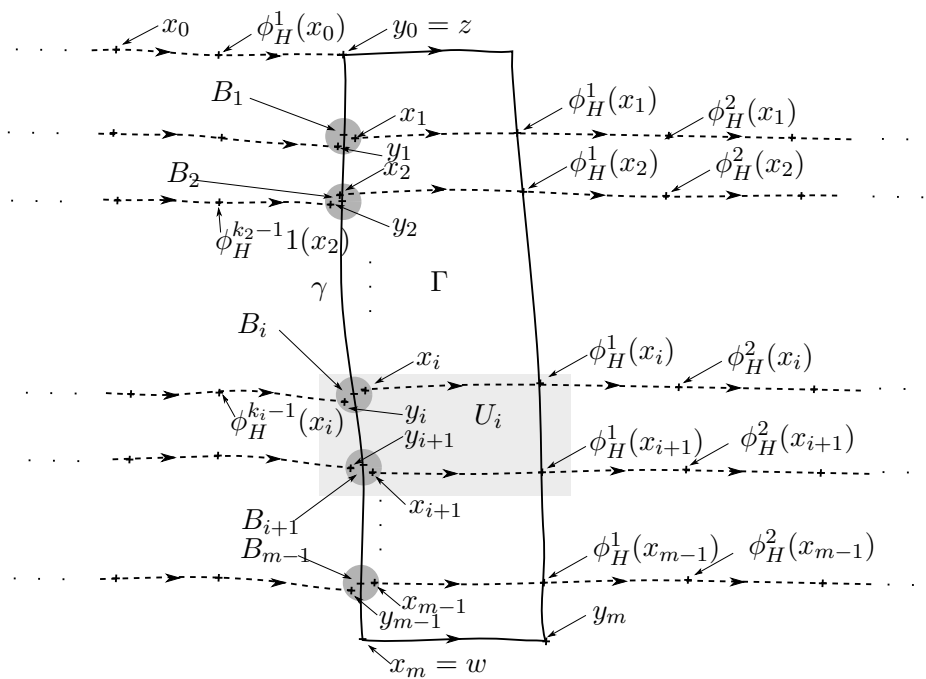
\caption{Settings from Step 1.}
\label{fig:Gamma}
\end{figure}

 \vspace{.2 cm}
 
\noindent \textbf{Step 2. Turning the sequence $x_0, \ldots, x_m$ into a genuine orbit.} For each $0 \leq i \leq m-1$ pick smooth embedded curves $a_i:[0,1] \rightarrow M$ such that 
\begin{itemize} 
\item $a_i(0) = y_i$ and $a_i(1) = x_{i+1}$,
\item the diameter of $a_i$ is smaller than $\delta$ (note that $d(y_i, x_{i+1}) < \delta$),
\item the image of $a_i$ is disjoint from the image of $a_j$, for $j\neq i$, and it is contained in the open set $U_i$,
\item for each $i$ the image of $a_i$ does not intersect any of the curves $t \mapsto \phi^t_H(x_j), \;  t \in [-1, k_{j}+1]$, for any value of $j$, except at the two points $y_i = \phi_H^{k_{i}}(x_i)$ and $x_{i+1} =  \phi^0_H(x_{i+1})$. 
 Moreover, for each $ i  $ the image of $ a_i $ does not intersect the curves $ (\phi^t_H(x_0))_{t \in [-2, 0]}$ and $(\phi^t_H(x_m))_{t \in [0, 2]}$, except at the point $ x_0 $ and $ x_m $, respectively. 
\end{itemize}

The above conditions, combined with the fact that the points $x_i$ were picked such that $y_i$ and $x_{i+1}$ are not on the same trajectory of $\phi^t_H$, guarantee that for each $0 \leq i \leq m-1$ we can find $W(a_i)$, a small neighborhood of the image of $a_i$, with the following properties: $W(a_i) \subset U_i$, the diameter of $W(a_i)$ is less than $\delta$, these neighborhoods are disjoint, and  each of the curves $(\phi^t_H(y_i))_{t\in [{ -1,1}]}$ and $(\phi^t_H(x_{i+1}))_{t\in[{ -1, 1}]}$ takes values in $W(a_i)$ if and only if $t \in (-2 \kappa_i, 2 \kappa_i)$ where $\kappa_i>0$ is small. 
 Now, let $\kappa = \min \{\kappa_1, \ldots, \kappa_{m-1} \}.$ By shrinking the neighborhoods  $W(a_i)$ we may assume the following: \\ If $t \in [-1, 1]$, then  for $ 0 \leqslant i \leqslant m-1 $ we have 
\begin{align}\label{eq:image_in_W}
  \phi^t_H(y_i) \in W(a_i) \iff \phi^t_H(x_{i+1}) \in W(a_i) \iff t \in (-2\kappa,2\kappa).
 \end{align}

If $t \in [-2, 1]$, then
\begin{align}\label{eq:image_in_W-2}
  \phi^t_H(x_{0}) \in W(a_0) \iff t \in (-2\kappa,2\kappa).
 \end{align}
If $t \in [-1, 2]$, then
\begin{align}\label{eq:image_in_W-3}
  \phi^t_H(x_{m}) \in W(a_m) \iff t \in (-2\kappa,2\kappa). 
 \end{align}

Next, we find a Hamiltonian, say $G_1$, the time-1 map of whose flow we will denote by $\tau$, such that $G_1$ is supported in the union of $W(a_i)$'s, and 
\begin{equation}\label{eq:def_tau} \tau(\phi_H^t(y_i)) = \phi^t_H(x_{i+1}), \;  \forall \;t \in [-\kappa, \kappa]. \end{equation}

Let $\psi_1 := \tau \phi^1_H$; observe that $\psi_1^{k_{i}}(x_i) = x_{i+1}$ for each $0 \leq i \leq m-1$.  Also, note that $d_{C^0}(\tau, Id) < \delta$ and thus $d_{C^0}(\psi_1, \phi^1_H) < \delta < \eps$.  Lastly, we remark that $G_1$ can be picked such that $\|G_1\|_{\infty}$ is as small as desired.  

\vspace{.2 cm}
\noindent \textbf{Step 3. A first approximation to  the invariant curve $\alpha$.} Let $k = k_0 + k_1 + \ldots + k_{m-1}$ and consider the smooth curve $\alpha_1 : [0,k+1] \rightarrow M$ given by the following formula: 
For $0 \leq i \leq m$, let $K_i = \sum_{j=0}^{i} k_j$ and define
$$ \alpha_1(t) :=  \begin{cases} 

 \phi^t_H(x_0)  \;\;\;\;\;\;\;\;\; \;\; \;\;\;\;t \in[0,1] \\

\tau \circ \phi^t_H(x_0)  \;\;\;\; \;\; \;\;\;\;t \in[1,K_0] \\

 \phi^{t - K_{i}}_H(x_{i+1}) \;\;\;\; \;\;\;\; t \in \left[K_{i},K_{i+1} -1 \right], 0\leq i \leq m-2\\

 \tau \circ \phi^{t - K_{i}}_H(x_{i+1}) \;\;\; t \in \left[K_{i+1} -1, K_{i+1} \right], 0\leq i \leq m-2\\

 \phi^{t-k}_H(x_m) \;\;\;\;\;\; \;\;\;\;\;\; t \in [k,k+1]
\end{cases}$$

  The fact that the above formula yields a smooth curve is an immediate consequence of Equation \eqref{eq:def_tau}. By slightly perturbing the $x_i$'s, while ensuring that the $ k_i $'s remain unchanged), and the Hamiltonian diffeomorphism $\tau$ if needed we can guarantee that this curve is embedded; for example we can pick the $x_i$'s from different level sets of $H$ and pick $\tau$ such that the curve $t \mapsto \tau(\phi_H^t(y_i))$, where $t \in [-2\kappa, 0]$, does not intersect  $\phi^t_H(x_{i+1})$ for $t  \in (0, 2 \kappa)$.

\begin{remark}\label{rem:alpha1_location}  In this remark, we will prove that the curve $\alpha_1$ satisfies the second property from the statement of the theorem.

We take $t_0 = K_0 - 2 \kappa = 2 - 2 \kappa$.  Then, it follows from the definitions of $\alpha_1$ and $\tau$, and the fact that $x_0 = \phi^{-2}_H(z)$, that  $\alpha_1(t) = \phi^{t-2}_H(z)$ for $t \in [0, t_0]$.  We remark that $\alpha_1([t_0, k+1])\subset U \cup \m C(\gamma((0,1)))$.  This can easily be checked from the definition of $\alpha_1$ and noting the following facts: $\tau$ is supported in $U$, and $\{x_i: 1 \leq i \leq m\} \subset U \cap \m C(\gamma((0,1))$. 

  In Step 4 we will have to modify $\alpha_1$ to obtain $\alpha$.  However, we will never change $\alpha_1$ on $[0,t_0]$ and all of the modifications will take place inside $U$.  Hence, the second property from the statement of the theorem will continue to hold for $\alpha$ as well.
\end{remark}
\begin{remark} \label{rem:third_property} Note that $\alpha_1(k+s) = \phi^s_H(w)$, for $s \in[0,1]$, and  $ \psi_1 \circ \alpha_1(k+s) = \phi_H^{s+1}(w) $ for $ s \in[0,1] $.  So at this stage the third property is also satisfied for any choice of $V$ and any sufficiently small value of $b$.  Throughout the remainder of the proof we will have to modify the curve $\alpha_1$.  However, the modifications on $[k,k+1]$ will be  such that the the formula given in the third property will remain true.  Furthermore, we will also modify $\psi_1$ to a different map $\psi$.  Now, $\psi$ will be constructed such that the support of $\theta:= \psi \circ \phi_H^{-1}$ will not intersect the curve $\phi^1_H \circ \alpha([k, k+1])$.  It will then follow that  $ \psi \circ \alpha(k+s) = \phi_H^{s+1}(w) $ for $ s \in[0,\tfrac12-b]\cup[\tfrac12+b,1] $.  This establishes the third property from the statement of the theorem.
\end{remark}

\begin{remark}\label{rem:littleinvariance}
For $ 0 \leq l \leq k$, we let $\alpha_1|_{[l, l+1]}:[0,1] \rightarrow M$ denote the curve  $\alpha_1|_{[l, l+1]}(t):= \alpha_1(l+t), \; \forall t \in [0,1]$.  Note that  $\psi_1 \circ \alpha_1 |_{[l-1,l]} = \alpha_1|_{[l, l+1]}$ for $ l \notin \{K_i : 0 \leq i \leq m-1 \}.$  For the remaining values of $l$ the two curves $\psi_1 \circ \alpha_1 |_{[l-1,l]}$ and $\alpha_1|_{[l, l+1]}$ only coincide near their endpoints.  Indeed, it can easily be checked, using Equations \eqref{eq:image_in_W} and \eqref{eq:def_tau}, that  both of the above curves coincide with the curve $t \mapsto \phi^t_H(x_{i+1})$ on $ [0, \kappa] \cup [1-\kappa,1]$. 
\end{remark}

\begin{remark}\label{rem:both_curves_inUi}
Let $l=K_i, 0\leq i \leq m-1$. The two curves $\psi_1 \circ \alpha_1 |_{[l-1,l]}$ and $ \alpha_1|_{[l, l+1]}$ are both contained in $U_i$.  This follows from the construction of $\tau$ and the fact that the curves $t \mapsto \phi^t_H(x_i)$ and $t\mapsto \phi^t_H(y_i) $, $t\in [0,1]$, are both contained in $U_i$.  We leave the details of this to the reader.
\end{remark}

We will next check that for each $ l \in \{K_i : 0 \leq i \leq m-1 \}$ the $C^0$ distance between the two curves $\psi_1 \circ \alpha_1 |_{[l-1,l]}$ and $ \alpha_1|_{[l, l+1]}$ is small.  More precisely, we will prove: 

\begin{claim}\label{cl:curves_are_close_by}
$d( \psi_1 \circ \alpha_1(l-1+t) , \alpha_1(l+t)) < \eps, $  for each $t \in [0,1],$ and for each $l \in \{K_i : 0 \leq i \leq m-1 \}$.
\end{claim}

\begin{proof}[Proof of Claim]  
  Using the definition of $\alpha_1$, it can easily be checked that $\alpha_1(K_i+t) = \phi^t_H(x_{i+1})$ and $\alpha_1(K_i-1 +t)  =  \tau \phi^{k_i-1+t}_H(x_{i}) = \tau \phi^{-1+t}_H(y_{i})$.  It follows  that  $\psi_1 \circ \alpha_1(K_i-1+t)  = \tau \phi^1_H \tau \phi^{-1 + t}_H (y_{i})$.

Now, recall that the curves $t \mapsto \phi^t_H(y_{i})$ and  $t \mapsto \phi^t_H(x_{i})$ intersect the support of $\tau$ only for $t \in (-2 \kappa, 2 \kappa)$.   Using this fact it can be checked  that 

$$
\tau \phi^1_H \tau \phi^{-1 + t}_H (y_{i}) = 
\begin{cases} 
\tau \phi^t_H(y_{i}), & t\in [0, 2 \kappa],\\
\phi^t_H(y_{i}), & t\in [2 \kappa, 1-2\kappa],\\
\tau \phi^1_H ( \tau \phi^{-1 + t}_H (y_{i})), & t\in [1-2 \kappa, 1].

\end{cases} $$

If $ t\in [0, 2 \kappa]$ then $\tau \phi^t_H(y_{i}), \phi^t_H(x_{i+1}) \in W(a_{i})$  and the set $W(a_{i})$ has diameter less than $\delta$, thus $d(\tau \phi^t_H(y_{i}), \phi^t_H(x_{i+1})) < \delta < \eps$.

If $t \in [2 \kappa, 1-2\kappa]$, then we must prove that $d(\phi^t_H(y_{i}), \phi^t_H(x_{i+1})) < \eps$.  This follows immediately from Equation \eqref{eq:delta} and the fact that 
 $d(y_{i}, x_{i+1}) < \delta$.

Finally, we consider the case  $t\in [1-2 \kappa, 1]$:  Write $\phi^t_H(x_i) = \phi^1_H (\phi^{-1+t}_H(x_i))$.  Note that in this case $\tau \phi^{-1 + t}_H (y_{i})$ and $\phi^{-1 +t}_H(x_{i+1})$ are both contained in the set $W(a_{i})$ which has diameter less than $\delta$. We see, via Equation \eqref{eq:delta}, that  $$d(\tau \phi^1_H(\tau \phi^{-1 + t}_H (y_{i})), \phi^t_H(x_{i+1})) \leq $$ $$ d_{C^0}(\tau, Id) + d(\phi^1_H(\tau \phi^{-1 + t}_H (y_{i})), \phi^1_H(\phi^{-1+t}_H(x_{i+1}))) < \delta + \frac{\eps}{2} < \eps,$$ which proves our claim.
\end{proof}

\vspace{.2cm}

\noindent \textbf{Step 4. Constructing the curve $\alpha$.}  Let $\m V$ denote a contractible open subset of $M$. Since  $\m V$ is contractible there exists a 1-form $\lambda$ such that $\omega = d\lambda$ inside $\m V$. Using the 1-form $\lambda$ we can define the action of any  curve $a: [0,1] \rightarrow \m V$ to be the integral $\int_0^1 a^* \lambda.$  Of course, the action of a curve depends on $\lambda$, however the difference of action between two curves with the same endpoints does not depend on the choice of $\lambda$ or $\m V$.
 
 For any $0 \leq l \leq k$, the two curves  $\alpha_1|_{[l, l+1]}$ and $\psi_1 \circ \alpha_1|_{[l-1,l]}$ coincide near their endpoints.  Furthermore, one can find a contractible neighborhood of $\alpha_1|_{[l, l+1]}$ which contains $\psi_1 \circ \alpha_1|_{[l-1,l]}$:  Indeed, as mentioned in Remark \ref{rem:littleinvariance}, if $ l \notin \{K_i : 0 \leq i \leq m-1 \}$, then  $\psi_1 \circ \alpha_1 |_{[l-1,l]} = \alpha_1|_{[l, l+1]}$.  Hence, we can take  any sufficiently small neighborhood of $\alpha_1|_{[l,l+1]}$.  If $ l = K_i, \; 0 \leq i \leq m-1$, then by Remark \ref{rem:both_curves_inUi} both curves are contained in $U_i$ which is contractible. 
 
 If $ l=1 $ the two curves $\alpha_1|_{[l, l+1]}$ and $\psi_1 \circ \alpha_1|_{[l-1,l]}$ coincide.  Beginning with $l=2$, for each $ 2 \leq l \leq k $ we successively make  $C^0$--small perturbations of $\alpha_1|_{[l,l+1]}$ near $t= l+ \tfrac12$ to obtain an embedded curve $\alpha$ with the property that the two curves  $\alpha|_{[l,l+1]}$ and $\psi_1 \circ \alpha|_{[l-1,l]}$ have the same action and are still contained in the same contractible open set mentioned in the previous paragraph. We should add that it is a well-known fact that one can make an arbitrary adjustment to the action of a curve by performing a $C^0$--small perturbation; see for example Remark A.13 of \cite{BuOp}. See Figure \ref{fig:Gamma2}.

 \begin{figure}
\centering
\def\svgwidth{\textwidth}
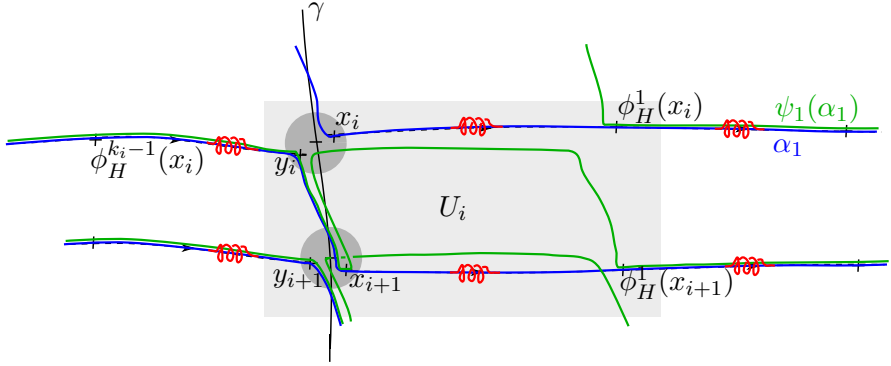
\caption{A portion of the curve $\alpha_1$ in blue, its image by $\psi_1$ in green, and the curve $\alpha$ obtained from $\alpha_1$ by local perturbations which are represented in red.
}
\label{fig:Gamma2}
\end{figure}

  We will ensure that the perturbations made in the previous paragraph  satisfy the following properties.  In the case where $ l \in \{K_i : 0 \leq i \leq m-1 \}$, we require the $C^0$ distance between $\alpha|_{[l,l+1]}$ and $\psi_1 \circ \alpha|_{[l-1,l]}$ to be less than $\eps.$  The fact that this can be achieved follows from Claim \ref{cl:curves_are_close_by} and by taking the perturbations from the previous paragraph to be sufficiently small.  Next consider the case where $ l \notin \{K_i : 0 \leq i \leq m-1 \}$.  In the previous paragraph, we obtained $\alpha$ from $\alpha_1$ by making a $C^0$ small perturbation of $\alpha_1$ in the interior of a small closed interval which contains the point $t = l + \frac{1}{2}$; let $I_l \subset [l, l+1]$ denote this interval. Since the intervals $I_l$ can be taken to be arbitrarily small and the action-adjusting perturbations from the previous step can be taken to have arbitrarily small support, we can pick contractible neighborhoods $V_l$ of  $\alpha|_{I_l}$ such that these neighborhoods are  pairwise disjoint and  the image of $\psi_1 \circ \alpha|_{I_{l-1}}$ is contained in $V_l$ for each $ l$.  Furthermore, we require that $\phi^1_H(V_l) \cap V_l = \emptyset$ and that $ \alpha(t) \in V_l $ only if $ t \in (l,l+1) $. For $ l = K_i $ we of course can assume that $ \psi_1 \circ \alpha|_{[l-1,l]}$ and $ \alpha|_{[l,l+1]}$ are contained in $ U_i $. We remark that by taking $V_k$ and $I_k$ to be sufficiently small we can ensure that the formula for $\alpha$ from  the third property of the statement of the theorem continues to hold. 
  
  Lastly, we point out that since 
  $ \alpha_1([2, k+1])$ is contained in $U \cup \m C(\gamma((0,1)))$ (see Remark \ref{rem:alpha1_location}), by picking the sets $V_l$ to be small enough, we may assume that each $V_l$ is contained in  $U \cup \m C(\gamma((0,1)))$. 
 
\vspace{.2 cm}

\noindent  \textbf{Step 5.  Turning $\alpha$ into an invariant curve.}  In this final step of the proof, we will perturb $\psi_1$ to a Hamiltonian diffeomorphism $\psi$ such that $\psi \circ \alpha|_{[l-1,l]} = \alpha|_{[l,l+1]}$.  
 
  In Step 4 we arranged that the two curves $\psi_1 \circ \alpha|_{I_{l-1}}$ and $\alpha|_{I_l}$ have the same action.  First assume that $ l \notin \{ K_i  :  0 \leqslant i \leqslant m-1 \}$. Then, we can find a  Hamiltonian diffeomorphism $\eta_l$, generated by a Hamiltonian which is compactly supported in $V_l$, such that $\eta_l \circ \psi_1 \circ \alpha|_{I_{l-1}} =  \alpha|_{I_l}$. Now, let $\eta$ be the composition of  the $\eta_l$'s  for $ l \notin \{K_i : 0 \leq i \leq m-1 \}$ and write $\psi_2 = \eta \circ\psi_1$ and observe that 
\begin{equation}\label{eq:invariant_almost} \psi_2 \circ \alpha|_{[l-1,l]} = \alpha|_{[l, l+1]}.
 \end{equation}
  
    Since the supports of $\eta_l$'s are disjoint and can be taken to be as small as one wishes, we obtain the following inequalities  $$  
    d_{C^0}(\psi_2, \phi^1_H) =  d_{C^0}(\eta\circ \psi_1, \phi^1_H) \leq  d_{C^0}(\eta, Id) + d_{C^0}( \psi_1, \phi^1_H) < \delta < \eps.$$  
    
    We remark that we can assume that $\eta$ is the time-1 map of a Hamiltonian, say $G_2$, whose norm $\|G_2\|_{\infty}$ can be made as small as one wishes.  Existence of such Hamiltonian $G_2$ follows from Lemma \ref{lemma:small_Ham_map_curves}.

  We will now deal with the case $ l \in \{K_i : 0 \leq i \leq m-1 \}$.  Fix $i$ and to simplify our notation denote $\gamma_0(t) := \psi_2 \circ \alpha|_{[K_i-1,K_i]}(t) = \psi_2 \circ \alpha(K_i-1+t)$ and $\gamma_1(t) := \alpha|_{[K_i,K_i+1]}(t) = \alpha(K_i+t).$
  The curves  $\gamma_0, \gamma_1$ coincide for $t\in [0, \kappa] \cup [1- \kappa,1]$: by Remark \ref{rem:littleinvariance} this was true for $\alpha_1$, and $\alpha$ differs from $\alpha_1$ only for $t$ near $K_i+\frac{1}{2}$.  Recall from Step 4 that the $C^0$ distance between  $\gamma_0$ and $\gamma_1$ is less than $\eps$ and that both of these curves are contained in $U_i$.
  \begin{claim}\label{cl:existence_small_homotopy}There exists a homotopy $F_i(s,t): [0,1] \times [\kappa,1-\kappa] \rightarrow U_i$ rel. end points from $\gamma_0|_{[\kappa, 1- \kappa]}$ to $\gamma_1|_{[\kappa, 1- \kappa]}$ such that under this homotopy the trajectory of any point of $\gamma_0|_{[\kappa, 1- \kappa]}$ has diameter less than $C \eps$  for some constant $C>0$ that only depends on $U$, the diffeomorphism $\Phi$ introduced in Step 1 and the Riemannian metric. 
 \end{claim}
 \begin{proof}
 The existence of the homotopy $F_i$ follows from the fact that the diffeomorphism $\Phi$ picked at the beginning of the proof identifies $U_i$ with the  box $(\frac{i}{m} - a, \frac{i+1}{m} +a )  \times (-c, 1+c) \times (-c,c) \times \ldots \times (-c,c)$.  We can take $F_i$ to be the straight-line homotopy in these coordinates.  The statement about diameter of trajectories of points on $\gamma_0|_{[\kappa, 1- \kappa]}$ follows from the fact that the $C^0$ distance between $\gamma_0$ and $\gamma_1$ is less than $\eps$.  The constant $C$ is given by the ratio of the pull back, via $\Phi$, of the  Riemannian metric $d$ to $(-c, 1+c )  \times (-c, 1+c) \times (-c,c) \times \ldots \times (-c,c)$ and the usual Euclidean metric.
 \end{proof}

  The dimension of $M$ is at least 4 and thus we can perturb the homotopy $F_i(s,t)$, provided to us by the above claim, to ensure that its image intersects the image of  $\alpha$ only for $t \in (K_i, K_i+1)$ and the image of $ \psi_2 \circ \alpha$ only for $ t \in (K_i -1,K_i)$.  Let $W_i \subset U_i$ be a small neighborhood of the image of the homotopy $F_i$ which intersects the images of $\alpha$ and $\psi_2 \circ \alpha$ only for $t \in (K_i, K_i+1)$ and $t \in (K_i-1, K_i)$, respectively.  
 Lastly, recall that we picked the sets $U_i$ in Step 1 such that $U_i$ can only intersect $U_{i-1}$ and $U_{i+1}$.  Hence, $W_i$ can only intersect $W_{i-1}$ and $W_{i+1}$.
  
  Now, we apply the $h$-principle of Proposition \ref{prop:h_principle} to obtain a Hamiltonian, say $G_i'$, generating a Hamiltonian isotopy $\varphi_i^t, t\in[0,1]$,  such that
  
  \begin{enumerate}
  \item $G_i'$ is supported in $W_i$  and $\|G_i'\|_{\infty}$ is as small as one wishes, 
  \item $\varphi_i^1 ( \psi_2 \circ \alpha|_{[K_i-1,K_i]}) =  \alpha|_{[K_i,K_i+1]}$, 
  \item $d_{C^0}(Id,  \varphi_i^1) < 2C \eps$.
  \end{enumerate}

Let $\phi_{\mathrm{odd}}$ denote the composition of the diffeomorphisms  $\varphi_i^1$ for odd $i$. Note that $\phi_{\mathrm{odd}}$ is the time-1 map of the Hamiltonian $G_{\mathrm{odd}}$ which is the sum of all $G_i'$ for odd $i$. Define $\phi_{\mathrm{even}}$ and $G_{\mathrm{even}}$ similarly.  Let $\psi := \phi_{\mathrm{even}} \phi_{\mathrm{odd}} \psi_2$.  

We remark that  $W_{m-1}$ is disjoint from  $\phi^1_H \circ \alpha([k, k+1])$. As pointed out in Remark \ref{rem:third_property}, this ensures that the third property from the statement of the theorem holds. 

We will now check that it is indeed true that $\psi \circ \alpha(t) = \alpha(t+1)$ for $t\in[0,k]$. If $l \notin \{K_i : 0 \leq i \leq m-1 \}$, then $\psi \circ  \alpha|_{[l-1,l]} = \psi_2 \circ \alpha|_{[l-1,l]}$: this is because we picked the sets $W_i$, whose union contains the supports of $\phi_{\mathrm{even}}$ and $\phi_{\mathrm{odd}}$, such that that they do not intersect the image of  $\psi_2 \circ \alpha|_{[l-1,l]}$.   We have already checked, in Equation \eqref{eq:invariant_almost}, that $\psi_2 \circ \alpha(t) = \alpha(t+1)$ for $t \in [l-1,l]$.  We must next check that $\psi \circ  \alpha|_{[l-1,l]} = \alpha|_{[l,l+1]}$ for $l \in \{K_i : 0 \leq i \leq m-1 \}$.  Fix $i$ and  let $l = K_i$. Then, $\phi_{\mathrm{even}} \phi_{\mathrm{odd}} \psi_2 \circ \alpha|_{[K_i-1,K_i]} = \varphi_i^1 \psi_2 \circ \alpha|_{[K_i-1,K_i]}$ because we picked the sets $W_i$ such that the image of $\psi_2 \circ \alpha|_{[K_i-1,K_i]}$ only intersects $W_i$. Now, $\varphi_i^1$ was picked such that   $\varphi_i^1 (\psi_2 \circ \alpha|_{[K_i-1,K_i]}) = \alpha|_{[K_i,K_i+1]}.$

\medskip

The rest of the proof is dedicated to verifying the fourth and the fifth properties from the statement of the theorem. We will first check that $d_{C^0}(\psi, \phi^1_H) < (4C+1) \eps$.  First, note that  $d_{C^0}( \phi_{\mathrm{odd}}, Id) < 2C \eps$ because $\phi_{\mathrm{odd}}$ is the composition of the diffeomorphisms  $\varphi^1_i$, for odd $i$, which have disjoint supports and each of which satisfies $d_{C^0}(\varphi_i^1, Id) < 2C \eps$.  Similarly, $d_{C^0}( \phi_{\mathrm{even}}, Id) < 2C \eps$.  It follows that
$d_{C^0}(\psi, \phi^1_H) =  \linebreak 
d_{C^0}( \phi_{\mathrm{even}} \phi_{\mathrm{odd}} \psi_2, \phi^1_H) \leq  d_{C^0}( \phi_{\mathrm{even}} \phi_{\mathrm{odd}}, Id) + d_{C^0}(\psi_2, \phi^1_H) < 4C \eps + \eps$; recall that we proved earlier that $d_{C^0}(\psi_2, \phi^1_H) < \eps$.  By going back to Step 1 and replacing $\eps$ with $\frac{\eps}{4C+1}$ we obtain $d_{C^0}(\psi, \phi^1_H) <  \eps$.


It is fairly easy to see that we obtained $\psi$ from $\phi^1_H$ by composing it, on the left, with the Hamiltonian diffeomorphism 
$\theta:= \phi_{even} \circ \phi_{odd} \circ \eta \circ \tau$; we denote by $F$ the generating Hamiltonian of $\theta$.  Now, the supports of the diffeomorphisms $\tau, \eta, \phi_{odd}, \phi_{even}$ are all contained in the union of $U$ and the open sets $V_l,  l \notin \{K_i : 0 \leq i \leq m-1 \}$. Recall that, as was mentioned at the end of Step 4, each of the sets $V_l$ is contained in $U \cup \m C(\gamma((0,1)))$.  This proves the claim about the support of $F$. 

As was remarked throughout the proof, the generating Hamiltonians of  $\tau, \eta, \phi_{odd}, \phi_{even}$, which  we denoted by $G_1, G_2, G_{\mathrm{odd}}, G_{\mathrm{even}}$, respectively, were picked to have norms as small as one wishes.  This proves the claim about $\|F\|_{\infty}$; we have established the fourth property in the statement of the theorem.  

To finish the proof of the theorem, it remains to show that $\psi$ has the same set of fixed points as $\phi^1_H$.  Now, $\psi = \theta \circ \phi^1_H$ and $\theta$ is supported in the union of $U$ and the open sets $V_l$.  Since $U$ and the $V_l$'s do not contain any of the fixed points of $\phi^1_H$ we conclude that the fixed points of $\phi^1_H$ are all fixed points of $\psi$ as well.  Now, to show that $\psi$ has no additional fixed points we will check that any point $x$ in the support of $\theta$ cannot be a fixed point of $\psi$.  First,  suppose that $x \in U$.  Since $U$ is a small neighborhood of the compact set $\Gamma([0,1] \times [0,1])$, which contains no fixed points of $\phi^1_H$, there exists a positive number, say $\varrho$, such that 
$d(x, \phi^1_H(x)) > \varrho$ for all $x \in U$.  Since $d_{C^0}(Id, \theta)  < \eps$, picking $\eps$ to be smaller than $\varrho$ guarantees that $\theta \phi^1_H(x) \neq x$. Next, suppose that $x \notin U$.  Hence, it must be the case that  $x$ is contained in one of the $V_l$'s.  We leave it to the reader to check, using the condition $\phi^1_H(V_l) \cap V_l = \emptyset$, that $x$ can not be a fixed point of $\psi$.  This completes the proof of Theorem \ref{theo:connecting-non-critical}.  
\end{proof}

\subsubsection{Proof of Theorem \ref{theo:invariant_curve}} \label{section:proof-connecting-2-critical}

In this section, we prove Theorem \ref{theo:invariant_curve} with the help of Theorems \ref{theo:invariant_curve_max_to_nearby} and \ref{theo:connecting-non-critical}.

Let $H$ be a Morse function as described in Lemma \ref{lemma:initial-morse-function}, and let $\eps>0$. Let $p, q$ be critical points of $H$, at least one of the two being non-extremal and let $\gamma:[0,1]\to M$ be a curve as in the statement of Theorem \ref{theo:invariant_curve}. We will give the proof under the assumption that $q$ is not a minimum. The proof in the case where  $q$ is a minimum (hence $p$ is not a maximum), can be easily adapted from this case. 

Let $\eps,\rho>0$, let $V_p$ be either the empty set if $p$ is a maximum or an open set containing $\{\phi_H^t(\gamma(0))\,:\,-\infty<t\leq 1\}$ is $p$ is not a maximum, and let $V_q$ be an open set containing $\{\phi_H^t(\gamma(1))\,:\,-1\leq t<+\infty\}$.

 Our construction will be carried out in three steps: first in the neighborhood of $p$, then from the neighborhood of $p$ to that of $q$, and finally in the neighborhood of $q$.

\medskip
\noindent\textbf{Step 1. Construction in the neighborhood of $p$.} We need to consider two cases. 

The simpler case is where $p$ is not a maximum. In this case, we set $z=\gamma(0)$ and $\alpha_1(t)=\phi_H^{t-2}(z)$ for all $t\in (-\infty, 1]$. At this point there is no need for a perturbation of $\phi_H^1$, thus we set $\theta_1=\id$. It is clear that, $\alpha_1$ and $\theta_1$ satisfy the requirements of Theorem \ref{theo:invariant_curve} for $\alpha$ and $\theta$.

Now consider the case when $p$ is a maximum. In this case, $\gamma(0)=p$ and it is our assumption that in some Darboux coordinates around $p$, $H$ is of the form $c\sum_{i=1}^n(x_i^2+y_i^2)$ where $c$ is a small negative constant. We denote by $\|\cdot\|$ the standard euclidean norm in these coordinates, and by $B(0,r)$ the euclidean ball centered at 0 and of radius $r$. Let $z$ be a point on the image of $\gamma$, distinct from $p$, but close enough to $p$ so that we may apply Theorem \ref{theo:invariant_curve_max_to_nearby} to the Hamiltonian $H'=-H$ and to the point $x=\phi_H^{-2}(z)=\phi_{H'}^2(z)$. Given a ball $B$ centered at $p$ and containing $x$, and given $\eps>0$, Theorem \ref{theo:invariant_curve_max_to_nearby} provides us with a curve $\alpha':[ - 1,+\infty)\to B \setminus \{p\}$ and a Hamiltonian homeomorphism $\theta'$ satisfying : 
\begin{itemize}
\item $\alpha'(t) =\phi^t_{H'}(x)$ for $t \in [- 1,0]$ and $\alpha'(t) \to p$ as $t\to+\infty$,
\item $\theta' \circ \phi^1_{H'} (\alpha'(t)) = \alpha'(t+1)$ for all $t \in [-1, \infty)$,
\item $\theta'(p)=p$ and $\theta' \circ \phi^1_{H'}$ has the same fixed points as $\phi_{H'}^1$,
\item The support of $\theta'$ is contained in the union of the open ball $B(0,\|x\|)$ and the $\eps$-neighborhood of $\{\phi_{H'}^t(x)\,:\,t\in[1,2]\}$,
\item $d_{C^0}(Id, \theta') < \eps'$ and $\|F'\|_{\infty} \leq \rho$, where $F'$ denotes a continuous Hamiltonian  such that $\phi^1_{F'} = \theta'$,
\item  For any neighborhood of $p$, there exists a compactly supported Hamiltonian diffeomorphism of $B$ which coincides with $\theta'$ in the complement of that neighborhood.
\end{itemize}
 

Now define $\alpha_1(t)=\theta'^{-1}(\alpha'( -t))$ for all $t\in(-\infty, 1]$ and $\theta_1=\theta'^{-1}$. We deduce the following from the above properties:
\begin{enumerate}
\item $\alpha_1(t) =\phi^{t-2}_{H}(z)$ for $t \in [0,1]$ and $\alpha_1(t) \to p$ as $t\to -\infty$, 
\item $\theta_1 \circ \phi^1_{H} (\alpha_1(t)) = \alpha_1(t+1)$ for all $t \in (-\infty, 0]$,
\item $\theta_1 \circ \phi^1_{H}$ has the same fixed points as $\phi_H^1$,
\item $\alpha_1$ takes values in the closed punctured ball centered at $p$ and having $z$ on its boundary, which is contained in $\m C(\gamma((0,1)))$.
\item The support of $\theta_1$ is contained in the union of the open ball $B(0,\|x\|)$ and the $\eps$-neighborhood of $\{\phi_{H}^t(x) \,:\,t\in[-2,-1]\}$, 
\item $d_{C^0}(Id, \theta_1) < \eps$ and $\|F_1\|_{\infty} <\rho$,  where $F_1$ denotes a continuous Hamiltonian  such that $\phi^1_{F_1} = \theta_1$,
\item  For any neighborhood of $p$, there exists a compactly supported Hamiltonian diffeomorphism of $B$ which coincides with $\theta_1$ in the complement of that neighborhood.
\end{enumerate}
We see that $\alpha_1$ and $\theta_1$ are compatible with the requirements of Theorem \ref{theo:invariant_curve}.

\medskip
\noindent\textbf{Step 2. From a neighborhood of $p$ to a neighborhood of $q$.} In both cases above we have set $z=\gamma(a)$ with $0\leq a< 1$ ($a>0$ if $p$ is a maximum, $a=0$ otherwise), and let us now set $w=\gamma(1)$ and $\tilde\gamma(t)=\gamma(a+(1-a)t)$ for $t\in[0,1]$. The curve $\tilde\gamma$ satisfies $\tilde\gamma(0)=z$ and $\tilde\gamma(1)=w$. It also satisfies that $ \frac{d}{dt} H \circ \tilde\gamma(t) < 0 $, which implies in particular that if $H$ is $C^2$-small enough then the map $\Gamma:(s,t)\mapsto \phi_H^t(\tilde\gamma(s))$ is a smooth embedding of the set $(\{0\}\times[-2,1])\cup((0,1)\times[0,1]) \cup(\{1\}\times[-1,2])$. In the case where $p$ is a maximum, note that since the flow of $H$ preserves level sets, the compact set $\Gamma([0,1]\times[0,1])$ does not intersect the open ball $B(0,\|x\|)=B(0,\|z\|)$. Moreover, if $\eps$ and $c$ are small enough, $\Gamma([0,1]\times[0,1])$ does not intersect the $\eps$-neighborhood of $\{\phi_{H}^t(x)\,:\,t\in[-2,-1]\}$ either. Hence, $\Gamma([0,1]\times[0,1])$ does not intersect the support of $\theta_1$.
  
We are now in the situation to apply Theorem \ref{theo:connecting-non-critical}: Let $U \supset \Gamma([0,1] \times [0,1])$ be an open set chosen so small that: 
\begin{itemize}
\item If $p$ is not a maximum, 
  \begin{equation}\label{eq:a}
    U\cup\,\m C(\tilde{\gamma}((0,1)))=U\cup\,\m C(\gamma((0,1)))\subset V_p\cup \m C(\gamma((0,1)))\cup V_q,
  \end{equation}
\item 
If $p$ is a maximum, $U$ does not intersect the support of $\theta_1$ and 
\begin{equation}\label{eq:b}
  U\cup\,\m C(\tilde{\gamma}((0,1)))= U\cup\,\m C(\gamma((a,1)))\subset \m C(\gamma((0,1)))\cup V_q.  
\end{equation}
\item $U$ does not intersect the piece of orbit $\{ \phi_H^t(w) : t \in [1+\frac{1}{3},+\infty)\}$.
\end{itemize}



Then, for every $\eps,\rho>0$, we can find a curve $\tilde\alpha_2:[0, k+1]\to M$, with $\alpha_2(0)= x = \phi_H^{-2}(z)$ and $\alpha_2(k)=w$, and a Hamiltonian diffeomorphism $\psi$ which is $C^0$-close to $\phi_H^1$ and satisfies a certain list of properties. We set $\theta_2=\psi\circ\phi_H^{-1}$. The properties of $\alpha_2$ and $\psi=\theta_2\circ\phi_H^1$ listed in Theorem \ref{theo:connecting-non-critical} have the following consequences:
\begin{enumerate}
\item $\alpha_2(t)=\alpha_1(t)=\phi_H^{t-2}(z)$ for all $t\in[0,1]$, 
\item $\theta_2\circ\phi_H^1(\alpha_2(t))=\alpha_2(t+1)$ for all $t\in[0,k]$,
\item $\theta_2\circ\phi_H^1$ has the same set of fixed points as $\phi_H^1$,
\item There exists $t_0\in(1,2)$ such that $\alpha_2(t)=\phi_H^t(x)$ for all $t\in[0,t_0]$ and $\alpha_2([t_0,k+1])\subset U\cup\,\m C(\tilde{\gamma}((0,1)))$. If $p$ is not a maximum then $\{\phi_H^t(x)\,:\,t\in[0,t_0]\} \subset V_p$ by definition of $V_p$; if $p$ is a maximum, then $\{\phi_H^t(x)\,:\,t\in[0,t_0]\}$ is contained in the boundary of the ball $B(0,\|x\|)$, hence in $ \m C(\gamma((0,1)))$. Using the inclusions \eqref{eq:a} and \eqref{eq:b}, we thus get that
 $\alpha_2$ takes values in $V_p\cup \m C(\gamma((0,1)))\cup V_q$.
\item $\theta_2$ is generated by a Hamiltonian $F_2$ which is supported in the open set $U\cup\,\m C(\tilde{\gamma}((0,1)))$, which does not intersect the support of $\theta_1$,
\item $d_{C^0}(\theta_2,\id)<\eps$ and $\|F_2\|_{\infty}<\rho$, 
\item $\theta_2$ is a smooth Hamiltonian diffeomorphism.
\end{enumerate}

All the points above follow immediately from Theorem \ref{theo:connecting-non-critical}.
Note that it also follows from Theorem \ref{theo:connecting-non-critical} that for any neighborhood $V$ of $\phi_H^{\frac12}(w)$, the curve $\alpha_2$ can be chosen so that for some $b>0$,  and all $t\in[k,k+1]$,
\begin{equation}\label{eq:1}
\alpha_2(t)\begin{cases}=\phi_H^{t -k}(w), & t\in[0,k+\tfrac12-b]\cup[k+\tfrac12+b,1]\\
\in V, & t\in[k+\tfrac12-b,k+\tfrac12+b].\end{cases}
\end{equation}

\medskip
\noindent{\textbf{Step 3. Construction in the neighborhood of $q$.}
By assumption, $\phi_H^t(w)=\phi_H^t(\gamma(1))\to q$ when $t\to+\infty$. Let $\tilde\alpha_3:[k,+\infty)\to M$ be the smooth curve defined by 
$$\tilde\alpha_3(t)=\begin{cases}\alpha_2(t), & t\in[k,k+1)\\
\phi_H^{t-k}(w), & t\in[k+1,+\infty).
\end{cases}$$ 
It has the following properties:
\begin{itemize}
\item $\tilde\alpha_3(t)\to q$ when $t\to+\infty$,
\item For some small $\kappa>0$, the image of $\tilde\alpha_3|_{[k+1+\kappa,+\infty)}$ does not intersect the supports of $\theta_1$ nor $\theta_2$,
\item For all $t\in [k,+\infty)\setminus(k+\tfrac12\!-\!b,k+\tfrac12\!+\!b)$,  $$\theta_2\circ\theta_1\circ\phi_H^1(\tilde\alpha_3(t))=\phi_H^1(\tilde\alpha_3(t))=\tilde\alpha_3(t+1),$$
\item According to \eqref{eq:1}, for $t\in[k+\tfrac12\!-\!b,k+\tfrac12\!+\!b]$, $\theta_2\circ\theta_1\circ\phi_H^1(\tilde\alpha_3(t))$ belongs to $\phi_H^1(V)$ which is an arbitrarily small neighborhood of $\tilde\alpha_3(k+1+\tfrac12)$.
\end{itemize}

For every integer $l\in\{k+1, k+2, \ldots \}$, we let $B_l$  be a small ball centered at $\tilde\alpha_3(l+\tfrac12)$ and included in the open set $V_q$. We may assume $V$ is small enough so that $V\subset U$ and $\phi_H^1(V)\subset B_{k+1}$. We also assume that the balls $B_l$ are all disjoint, and that they do not intersect the supports of $\theta_1$ nor $\theta_2$. Similarly as in Step 4 of the proof of Theorem \ref{theo:connecting-non-critical}, we successively make $C^0$-small perturbations of $\tilde\alpha_3|_{[l,l+1]}$ in $B_l$, for $l\in\{k+1, k+2, \ldots \}$, to obtain an embedded curve $\alpha_3$, such that the curves $\phi_H^1(\alpha_3|_{[l-1,l]})$ and $\alpha_3|_{[l,l+1]}$ have the same action for all $l$. We can choose those perturbations so that, for each $l$, these two curves
 coincide in the complement of a very small interval $J_l$ centered at $l+\frac12$, and they both send $J_l$ into $B_l$. 
Using Lemma  \ref{lemma:small_Ham_map_curves} (as in Step 5 in the proof of Theorem \ref{theo:connecting-non-critical}), this property implies that we can find Hamiltonian diffeomorphisms $\eta_l$, generated by Hamiltonians which are compactly supported in the $B_l$'s, such that $\eta_l(\phi_H^1(\alpha_3(t)))=\alpha_3(t+1)$ for all $t\in[l-1,l]$. Moreover, Lemma  \ref{lemma:small_Ham_map_curves} tells us that these generating Hamiltonians can be chosen with arbitrary small $\|\cdot\|_\infty$ norm.

Let $\theta_3$ be the composition of all the $\eta_l$'s for $l\in\{k+1, k+2, \ldots \}$. By construction, $\theta_3$ is generated by a Hamiltonian $F_3$ whose support does not intersect the support of $\theta_1$ nor that of $\theta_2$.

 Moreover, the $\eta_l$'s can be chosen $C^0$-close to $\id$ by shrinking the balls $B_{k+1}$, $B_{k+2}$, $\dots$. If we shrink each ball $B_l$ so that $\mathrm{diam}B_l<\inf_{z\in B_l}d(z,\phi_H^1(z))$, then $\theta_3\circ\phi_H^1$ has no fixed point in $B_l$. Indeed, the triangle inequality yields for all $z\in B_l$ : $$d(\theta_3\circ\phi_H^1(z),z)\geq d(\theta_3\circ\phi_H^1(z),\phi_H^1(z))-d(z,\phi_H^1(z))>0.$$
Therefore, the following list of properties is satisfied:
\begin{enumerate}
\item $\alpha_3(t)\to q$ when $t\to +\infty$, 
\item $\theta_3\circ \phi_H^1(\alpha_3(t)) 
=\alpha_3(t+1)$ for all $t\in[k,+\infty)$,
\item $\theta_3\circ \phi_H^1$ has the same set of fixed points as $\phi_H^1$,
\item For all $t\in [k,+\infty)$, $\alpha_3(t)\in V_q$,
\item $\mathrm{supp}(F_3)\subset V_q\cup\{q\}$,
\item $d_{C^0}(\id,\theta_3)<\eps$ and $\|F_3\|_{\infty}<\rho$.
\item For any neighborhood of $q$, $\theta_3$ coincides with a Hamiltonian diffeomorphism in the complement of that neighborhood.
\end{enumerate}

\begin{figure}[h!]
\centering
\def\svgwidth{\textwidth}
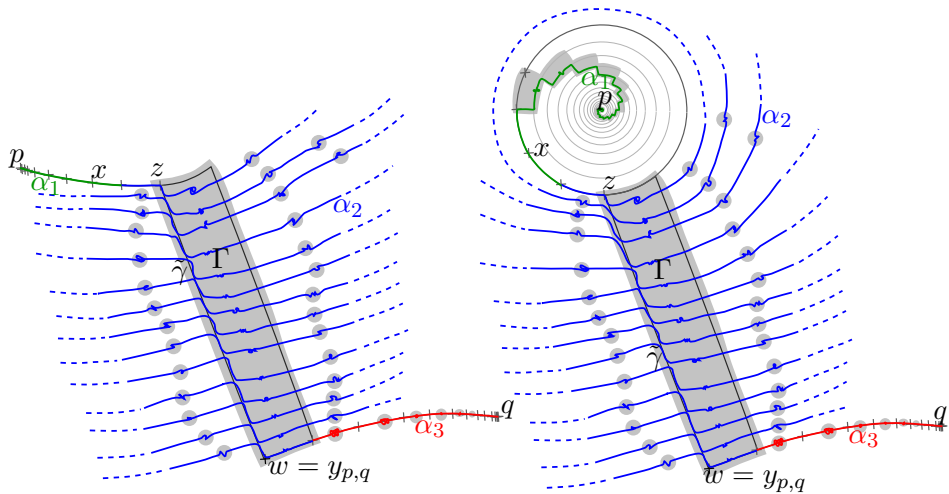
\caption{The case ``$p$ not a maximum'' is on the left, and the case ``$p$ a maximum'' is on the right. The curves $\alpha_1$, $\alpha_2$ and $\alpha_3$ are respectively in green, blue and red. The supports of the successive perturbations on $\phi_H^1$ are included in the grey regions.}
\label{fig:Alpha}
\end{figure} 

\medskip
\noindent{\textbf{End of the proof.}}
Let $\theta$ be the Hamiltonian homeomorphism $\theta=\theta_1\circ\theta_2\circ\theta_3$ and let $\alpha$ be the smooth curve which coincides with $\alpha_1$ on $(-\infty,1]$, with $\alpha_2$ on $[0,k+1]$, and with $\alpha_3$ on $[k,+\infty)$. The respective properties 1 to 7 established for each $\theta_1$, $\theta_2$ and $\theta_3$ in the previous steps, together with the fact that their supports are all disjoint, imply the corresponding properties  1 to 7 in Theorem \ref{theo:invariant_curve}} whose proof in now achieved. ~~~$\Box$

%
%

\bibliographystyle{abbrv}
\bibliography{biblio}


{\small

\medskip
\noindent Lev Buhovski\\
School of Mathematical Sciences, Tel Aviv University \\
{\it e-mail}: levbuh@post.tau.ac.il
\medskip

\medskip
\noindent Vincent Humili\`ere \\
Sorbonne Universit\'es UPMC Univ Paris 06. Institut de Math\'ematiques de Jussieu-Paris Rive Gauche.
UMR 7586, CNRS. Univ Paris Diderot, Sorbonne Paris Cit\'e. F-75005, Paris, France.
{\it e-mail:} vincent.humiliere@imj-prg.fr
\medskip

\medskip
 \noindent Sobhan Seyfaddini\\
 Institute for Advanced Study, School of Mathematics, Princeton, NJ.

\noindent CNRS, Institut de Mathématiques de Jussieu-Paris Rive Gauche, UMR7586, Sorbonne Universités, UPMC Univ Paris 06, Univ Paris Diderot, Sorbonne Paris Cité, F-75005, Paris, France. 
 {\it e-mail:} ssey@ias.edu

}

\end{document}